\documentclass[11pt%
,onecolumn%
]{IEEEtran}
\usepackage{amssymb,amsmath,theorem}
\usepackage{rotating,graphics,epsfig,epic}
\usepackage[dvipdfm]{color}
\usepackage{multicol,fancyhdr}

\definecolor{gray3}{rgb}{.67,.62,.62}

\newcommand{\rref}[1]{(\ref{#1})}
\newcommand{\qed}{\hfill $\blacksquare$} 

{\theorembodyfont{\rmfamily}
\newtheorem{definition}{Definition}
\newtheorem{theorem}{Theorem}
\newtheorem{corollary}{Corollary}
\newtheorem{assumption}{Assumption}
\newtheorem{remark}{Remark}
\newtheorem{example}{Example}

}
\renewenvironment{proof}[1][\!\!]{\noindent {\em Proof #1. }}{\qed}

\def\ie{{\it i.e.}}\def\eg{{\it e.g.}}
\def\cf{{\it cf.\ }}
\def\bfx{\mbox{\boldmath $x$}}
\def\bfy{\mbox{\boldmath $y$}}
\def\bfz{\mbox{\boldmath $z$}}
\def\bfe{\mbox{\boldmath $e$}}
\def\bfh{\mbox{\boldmath $h$}}
\def\bff{\mbox{\boldmath $f$}}
\def\bfd{\mbox{\bf d}}
\newcommand\mR{\mathbb{R}}\newcommand\mZ{\mathbb{Z}}
\def\S{\mathcal S}\def\cK{\mathcal K}
\def\I{\mathcal I}\def\cL{\mathcal L}
\def\N{\mathcal N}\def\cKL{\cK\cL}

\newcommand{\address}{%
$^{**}$Control of Complex Systems Laboratory, Institute of 
Problems of Mech. Engg., Bolshoi av., 61, V.O., St-Petersburg, 
199178, Russia.  E-mail: efde@mail.ru
   $^*$C.N.R.S, UMR 8506, Laboratoire de Signaux et Syst\`emes,
SUPELEC, Plateau de Moulon, 91192 Gif s/Yvette, France. 
E-mail: panteley@lss.supelec.fr, loria@lss.supelec.fr
}

\begin{document}

\title{\ \\[-5mm] \LARGE Robust output stabilization: improving performance \\ via supervisory control\thanks{{\address}}}
\author{\large
Denis Efimov,$^{**}$\quad 
Elena Panteley,$^*$\quad 
Antonio Lor\'{\i}a$^*$
}

\date{}%
\parskip = 4pt

 \maketitle

 \begin{abstract}
We analyze robust stability, in an input-output sense, of switched stable systems. The primary goal (and contribution) of this paper is to design switching strategies to guarantee that input-output stable systems remain so under switching. We propose two types of {\em supervisors}: dwell-time and hysteresis based. While our results are stated as tools of analysis they serve a clear purpose in design: to improve performance. In that respect, we illustrate the utility of our findings by concisely addressing a problem of observer design for Lur'e-type systems; in particular, we design a hybrid observer that ensures ``fast'' convergence with ``low'' overshoots. As a second application of our main results we use hybrid control in the context of synchronization of chaotic oscillators with the goal of reducing control effort; an originality of the hybrid control in this context with respect to other contributions in the area is that it exploits the structure and chaotic behavior (boundedness of solutions) of Lorenz oscillators.
 \end{abstract}

\section{Introduction}

There exist many good reasons and practical motivations to use a set of controllers for a single plant as opposed to one controller: for instance, we may think of a complex system whose dynamic behavior can only be described satisfactorily, from a practical viewpoint, by using several models, each corresponding to a {\em mode} of the system or, in other words, each model being valid for state values in a specified region of the state space. In such case, it is common practice to implement a group of ``local'' controllers applied depending on the operating mode. Another common situation is that when design {\em constraints} are imposed by certain ``(sub)optimality'' goals: for instance, one may ask for a controlled plant to track an operating point under constraints regarding transient performance, speed of convergence (to the desired operating point), robustness with respect to uncertainties or measurement noise, {\em etc.} The term {\em sub}optimality is therefore understood in that sense: to obtain, for instance, the fastest speed of convergence among a set of possibilities (not necessarily {\em all} possibilities). 

In such a scenario, it may result convenient to use a set of controllers, each of which achieves the control goal by respecting {\em one } of the constraints. The natural question which arises next is how to trade off the advantages of each ``good'' controlled system. For instance, given a controller designed to achieve (the) fast(est) convergence and given another one which ensures stabilization with (the) small(est) overshoots, how to design a switching rule between the two controllers to achieve both, (relatively) fast convergence at a low price (small overshoots). We address this problem via  so-called {\em supervisory} control  --{\em cf.} \cite{M1,M2} and \cite[Chapter 6]{LIBBOOK}; that is, the goal is to the design an algorithm to appropriately {\em switch} between local {\em given} controls depending on their properties, thereby combining their characteristics to improve performance and possibly to achieve ``certain optimality'' over larger domains. 

For the case of two controls --one local and another global-- the problem of uniting controllers has different general solutions and motivations regarding stability and robustness --{\em cf. } \cite{TEEECC97,P1,PP,IF,E1}. While theoretically challenging in general, focused problems where switching between a local and a global controller brings solutions otherwise difficult or even impossible to achieve are common in the control of mechanical systems: see \cite{al:SACTA,LEFECC97} for the problems of set-point and tracking control of robot manipulators under switching, \cite{SPO2}  for the problem of ``swinging up the acrobot'', \cite{ASTPRI,HESMORNH} for a hybrid schemes to control chained-form systems; {\em etc.}

On a more general basis, switching among more than two controllers imposes significant challenges to analysis since classic stability theory does not apply. For instance, as is well-known, switching between stable modes of a closed-loop control system does not necessarily yield a stable behavior --{\em cf.} \cite{LIBBOOK}. Over the past years there has been an exponentially increasing interest on stability and stabilization of switched systems, notably for linear switched systems --\cf \cite{LINANT,HOUMICHYE}. An important trend is that based on Lyapunov-like methods as for instance, finding a common Lyapunov function --\cf \cite{BRANICKY98,MASSHO,PELDEC}; invariance principles --\cf \cite{GOESANTEE,HESTACLASALLE,HESLIBANGSON}; geometric methods --\cf\, \cite{MAN2000} to mention a few. 

While it stands as a fact that stable systems do not remain stable under totally arbitrary switching, an important aspect of stability theory for switched systems focusses on the analysis and design of the switching regime. In the context of supervisory control designing the ``switching rule'' comes to orchestrate a collection of (stable) nonlinear control systems. Among the range of conditions one can think of to impose on the switching regime in order to establish stability, one of the most notable is dwell time --\cf\, \cite{MECC93,M1} and its many variants (average, weak, state-dependent, {\em etc.}) --\cf\, \cite{HM,PSM,HLM}. Roughly, the introduction of dwell time in the switching rule  comes to imposing a minimal amount of time that it is imposed to each mode to remain active. Following a similar train of thought, we may mention another method which aims at analyzing switched systems as time-varying systems with the property persistency of excitation of the switching signal --\cf\, \cite{al:CDC07CHITOUR}.

Implicitly, in the previous paragraphs we have discussed the  stability problem for switched systems with respect to the trivial {\em state} motion $\bfx(t)\equiv 0$. However, irrespectively of the switching method, the control goal may vary: from pure stabilization to robustness with respect to external inputs --{\em cf.} \cite{VCL,XWL} or measurement noise --{\em cf.} \cite{P1} to the more general setting of output stabilization --\cf\, \cite{PAVLOVbook,BYRISI}. In such case, the stabilization goal is to bring an {\em output} motion close to a desired operating point. Moreover, it appears natural to consider robustness aspects hence to study stability in an input-output sense. This is the context in which we place our main results.

For a given set of ``local'' controllers, each achieving a control goal in one dynamic mode, we address the problem of designing a supervisor achieving stability in an input-output sense that is close to that of \cite{SONWANIOS}, while improving performance with respect to each controlled mode. Each closed-loop system (say a plant and a local controller) is assumed to have a common output to be controlled and for which certain ``optimality'' is imposed (fast convergence, small overshoots, robustness, {\em etc.}). Stability is measured taking into account external disturbances, as for instance in \cite{HESLIBTEE} yet, with respect to an {\em output} motion as opposed to the full state. More precisely, our definition of input-output stability is given in terms of $\mathcal K\mathcal L$ estimates and a generalized norm for the inputs which includes the usual Euclidean and Lebesgue norms. As an aside, the definition of input-to-output state stability that we use covers other definitions previously proposed in the literature. Under such setting we show that by appropriate switching (hysteresis-based and dwell-time-based) input-to-output stable systems conserve such property. 

The scenario previously described is similar (in spirit) to that studied in the recent paper \cite{HESCDC07} where the primary goal is to improve performance of switched control systems; in this reference it is assumed that the switching rule is {\em given} and the freedom left to the designer resides on setting adequately the initial conditions of the controllers. Another recent article along similar lines is \cite{ZHEWU} where the authors consider the more particular case of switching between {\em two} controlled systems, one global backstepping controller and one local linear-parameter-varying. 

In contrast to works on {\em optimal} control of switched systems --\cf\, \cite{SEUCORGIUBEM} and references therein, we do not derive (sub)optimal control laws but rather, we assume that a set of controllers is given and then, we design a supervisor to orchestrate the switching among these ``local'' controls. Freedom is left in the way the output motion space is partitioned to assign a controller to each of the identified modes. The latter is done on a case-by-case basis, following a rule of thumb which is what is often done in control practice. 

We illustrate the latter point, and thereby the utility of our theoretical findings, by addressing the problem of hybrid observer design for Lur'e-type systems and synchronization of chaotic (Lorenz) oscillators. In the first case we design a hybrid dwell-time observer with the ``optimization'' goal of reducing the time of convergence of estimation errors. For synchronization of chaotic oscillators we present a hysteresis-based supervisor which aims at reducing control effort, as measured both in absolute amplitude (peaks) and energy (integral of squared value). The latter is achieved by applying {\em no control} over certain regions of the state space \ie, using the fact that Lorenz systems have a complex attractor hence, the solutions tend to a bounded (relatively ``small'') domain without control effort. 

The rest of the paper is organized as follows: Section \ref{sec:prel} contains basic definitions and notation; in Section \ref{sec:problem} we lay the standing assumptions for our main results, which are presented in Section \ref{sec:main}. In Section \ref{sec:lure} we present an application of our findings to observer design and synchronization of chaotic oscillators. We conclude with some remarks in Section \ref{sec:conlusion}.

\section{Preliminaries}
\label{sec:prel}
As usual, a continuous function $\sigma :{\mR}_+ \to {\mR}_+ $ is of class $\cK$ if 
it is strictly increasing and $\sigma \left( {\,0} \right)=0$; additionally 
it is of class $\cK_\infty $ if it is also radially unbounded; a 
continuous function $\beta :{\mR}_+ \times {\mR}_+ \to {\mR}_+ $ is of class $\cK\cL$, if 
$\beta (\cdot ,t)\in \cK$ for any $t\in {\mR}_+ $, and $\beta (s,\cdot )$ is 
strictly decreasing to zero for any $s\in {\mR}_+ $. 

Consider a family of systems
\begin{equation}
\label{eq1}
\dot {\bfx}={\rm {\bff}}_i (\,{\rm {\bfx}},{\rm {\bf d}}\,),
\quad
{\rm {\bfy}}={\rm {\bfh}}(\,{\rm {\bfx}}\,),
\quad
i\in \I,
\end{equation}
where ${\rm {\bfx}}\in \mR^n$ denotes the state vector; ${\rm {\bf d}}\in \mR^m$ denotes a  
disturbance; ${\rm {\bfy}}\in \mR^p$ denotes an output and $i$ is an index taking values from the countable set $\I\subseteq\mZ_+$. Assume that the functions ${\rm {\bff}}_i :\mR^{n+l}\to \mR^n$, ${\rm {\bf 
h}}:\mR^n\to \mR^p$ are continuous and locally Lipschitz with respect to ${\rm 
{\bfx}}$, $i\in \I$. It is assumed that ${\rm {\bf d}}:{\mR}_+ \to \mR^m$ is 
Lebesgue measurable and essentially bounded, \ie, for all $t\ge 0$
\[
\vert \vert {\rm {\bf d}}\vert \vert _{\,[\,t_0 ,t\,)} =ess\,\sup \,\left\{ 
{\,\vert {\rm {\bf d}}(\,t\,)\vert ,\,\,t\in [\,t_0 ,t\,)\,} \right\}.
\]
For $t=+\infty$ we write $\|{{\bfd}}\| = \vert \vert {{\bfd}}\vert \vert _{[0,+\,\infty )} $. We denote by $\mR_{\mR^m}$ the set of functions such that $\vert \vert {{\bfd}}\vert |\,<+\,\infty $. Our definition of robust stability is stated in terms of the generalised norm $\S:\mR_{\mR^m} \times \mR^2\to {\mR}_+ $ defined, for any $t_0 \ge 0$ by
\begin{align}\label{2000}
  S\,[{{\bfd}},t_0 ,t]:=\,a\int\limits_{t_0 }^t {\omega (\vert {{ 
d}}(\tau )|)\,d\tau } + b \,\vert \vert {{\bfd}}\vert \vert _{[t_0 ,t)},
\\
a,\, b\geq 0,\quad a+b >0,\quad \omega \in \cK.
\end{align}
The set of essentially bounded functions ${{\bfd}}$ s.t. $S\,[{{ d}},0,+\,\infty ]$ $<+\,\infty $ is denoted as $M_{\mR^m} $ ($M_{\mR^m} \subset \mR_{\mR^m} )$. 

Let $i:{\mR}_+ \to \I$ be piecewise constant and continuous from the right then, the family of systems (\ref{eq1}) defines the following switched system
\begin{equation}
\label{eq2}
\dot {\bfx}={\rm {\bff}}_{i(\,t\,)} (\,{\rm {\bfx}},{\rm {\bf 
d}}\,),
\quad
{\rm {\bfy}}={\rm {\bfh}}(\,{\rm {\bfx}}\,).
\end{equation}
We say that the switching signal $i(t)$ has an average dwell-time $0<\tau _D <+\,\infty 
$ if between switches, for any time instants $t_2 \ge t_1 \ge 0$ it holds that
\[
N_{[t_1 ,t_2 )} \le N_0 +\frac{t_2 -t_1 }{\tau _D }
\]
 for an integer $1\le N_0 <+\,\infty $ and where $N_{[t_1 ,t_2 )} $ is the number of discontinuities (switches) of the signal $i(t)$ --{\em cf. } \cite{HM}. If the interval between any two 
switches is not less than $\tau _D $ then the switching signal is said to have dwell-time property and $N_0 =1$. The system (\ref{eq2}), for signal $i(t)$ with 
average dwell-time or simple dwell-time, has a finite number of switches on any 
finite-time interval and its solution is continuous and defined at least 
locally --\cf \cite{LIBBOOK}.
The switched system, for a switching signal $i(t)$, is called forward complete\footnote{ Necessary and sufficient conditions for a dynamical system (\ref{eq1}) to be forward complete with $S\,[{\rm {\bf d}},t_0 ,t]=\,\,\vert \vert {\rm {\bf d}}\vert \vert _{[\,t_0 ,t\,)} $ can be found in \cite{ANGSON-UFC}.} if for all initial conditions ${{\bfx}}_0 \in \mR^n$ and inputs ${{\bf d}}\in M_{\mR^m} $, the solutions ${\rm  {\bfx}}(t,{{\bfx}}_0 ,{{\bf d}})$ of the switched system (\ref{eq2}) are defined for all $t\ge 0$. We also denote the outputs by ${{\bfy}}(t,{\rm  {\bfx}}_0 ,{{\bf d}})={{\bfh}}({{\bfx}}(t,{{\bfx}}_0 ,{{\bf d}}))$. On occasions we may use the shorthand notation $\bfx(t)=\bfx(t,\bfx_0,\bfd)$, $\bfy(t)=\bfy(t,\bfx_0,\bfd)$.

 \begin{definition}
\label{Definition1} We say that, for a fixed $i\in \I$ a forward complete system (\ref{eq1}) is Input-to-Output Stable (IOS) with respect to the output $\bfy$ the input $\bfd$ and the norm $\S$ if, for all ${{\bfx}_0} \in \mR^n$ and $\bfd\in M_{\mR^m}$, there exist functions $\beta _i \in \cK\cL, \gamma _i \in \cK$ such that, for all $t\ge 0$,
\[
\vert {\rm {\bfy}}(\,t,{\rm {\bfx}}_0 ,{\rm {\bf d}}\,)\vert \,\,\le \beta 
_i (\,\vert {\rm {\bfx}}_0 \vert ,t\,)+\gamma _i (\,\S[\,{\rm {\bf 
d}},0,t\,]\,).
\]
We say  that a switched forward complete system (\ref{eq2}) with $i:{\mR}_+ \to \I$ is IOS with respect to the  output $\bfy$, the input $\bfd$ and the norm $\S$ if, for all $\bfx_0 \in \mR^n$ and $\bfd\in M_{\mR^m} $ there exist functions $\beta \in \cKL, \gamma \in \cK$ such that, for all $t \ge 0$,
\[ 
\vert {\rm {\bfy}}(\,t,{\rm {\bfx}}_0 ,{\rm {\bf d}}\,)\vert \,\,\le \beta 
(\,\vert {\rm {\bfx}}_0 \vert ,t\,)+\gamma (\,\S[\,{\rm {\bf d}},0,t\,]\,).
\]
\end{definition}
\begin{definition}
\label{Definition2} We say that for a fixed $i\in \I$ the forward complete system (\ref{eq1}) is state-independent IOS (SIIOS) with respect to the output $\bfy$, the input $\bfd$ and the norm $\S$ if, for all $\bfx_0 \in \mR^n$ and $\bfd\in M_{\mR^m} $, there exist functions $\beta'_i \in \cK\cL, \gamma'_i \in \cK$ such that, for all $t \ge 0$,
\[
\vert {\rm {\bfy}}(\,t,{\rm {\bfx}}_0 ,{\rm {\bf d}}\,)\vert \,\,\le 
{\beta }'_i (\,\vert {\rm {\bfh}}(\,{\rm {\bfx}}_0 \,)\vert ,t\,)+{\gamma 
}'_i (\,\S[\,{\rm {\bf d}},0,t\,]\,).
\]
We say that the switched forward complete system (\ref{eq2}) with $i:{\mR}_+ \to \I$ is SIIOS with respect to the output $\bfy$, the input $\bfd$ and the norm $\S$ if, for all $\bfx_0\in \mR^n$ and $\bfd\in M_{\mR^m} $, there exist functions $\beta'\in \cK\cL, \gamma'\in \cK$ such that, for all $t \ge 0$,
\[
\vert {\rm {\bfy}}(\,t,{\rm {\bfx}}_0 ,{\rm {\bf d}}\,)\vert \,\,\le 
{\beta }'(\,\vert {\rm {\bfh}}(\,{\rm {\bfx}}_0 \,)\vert ,t\,)+{\gamma 
}'(\,\S[\,{\rm {\bf d}},0,t\,]\,)
\]
The systems are exponentially SIIOS if $\beta'_i (s,r)=a\,s\,e^{-b\,r}$ or $\beta'(s,r)=a\,s\,e^{-b\,r}$ for some $a>0$, $b>0$. 
\end{definition}

For the case when $S\,[{\rm {\bf d}},t_0 ,t]=\,\,\vert \vert {\rm {\bf d}}\vert
\vert _{[\,t_0 ,t\,)} $, respectively $S$ is defined by \rref{2000}, closely connected input-output stability
properties and relations between them for nonlinear dynamical systems can be
found in \cite{SONWANIOS}; Lyapunov characterizations of these properties are
presented in \cite{SONWANIOS-LYA}. The difference between IOS and SIIOS
properties defined above strives in the dependence on initial conditions: for an
SIIOS system, if the initial amplitude ${\rm {\bfy}(t_0)}$ is ``small'' then the
amplitude ${\rm {\bfy}(t)}$ during transient remains ``small''; for IOS systems,
even for ``small'' initial amplitudes $\vert {\rm {\bfy}}(\,t_0 \,)\vert $ the
norm of ${\rm {\bfy}(t)}$ may be large during transients for ``large'' values of
$\vert {\rm {\bfx}}(\,t_0 \,)\vert $. However, the asymptotic behavior of IOS
and SIIOS systems is the same as they both ``forget'' about about the initial
conditions and converge to the set where the output is zero (in the case that
there are no external disturbances) or to some neighborhood of the latter
(where the size of the neighborhood is proportional to the disturbances
amplitude). For ${\rm {\bfy}}={\rm {\bfx}}$ both properties
are reduced to the well known properties of input-to-state stability (ISS) and integral ISS  properties --\cf\, \cite{SON1,IISS}.

\section{Problem statement}
\label{sec:problem}
Our main results rely on the following hypotheses. 
\begin{assumption}
\label{ass1}
{For each fixed }$i\in \I$,{ the system  (\ref{eq1}) is forward complete and SIIOS with respect to output }${{\bfy}}${ and input }${{\bf d}}${ with functions} $\beta _i \in \cK\cL$, $\gamma_i := \gamma \in \cK$ and norm $\S$. 
\end{assumption}
From Definition \ref{Definition2} it holds that $\beta _i (s,0)\ge s$ for all $s\ge 0$. {Without loss of generality we assume} that $\beta _i (s,0)>s$, $s>0$ and denote $\chi _i  (s)=\beta _i^{-1} (s,0)$.

The second standing assumption is that we dispose of an appropriate partition of the real line for normed output values and to each partition a stable system is assigned. It should be clear that the number of systems, $N$, and the number of partitions, $M$, is in general not the same. Let $q\in [0,M]$ define the strictly increasing sequence $\{\Delta_q\}$ and let each system from the family \rref{eq1} be labeled $\theta _q \in \I$ for each such $q$. 
{\begin{assumption}
\label{ass2}
 {A partition }$\mR_+ =\bigcup\limits_{q=0}^M {[\Delta _q ,\Delta _{q+1} )} 
, \Delta _0 =0, \Delta _{M+1} =+\,\infty ${ is given and, for each }$ q\in [0, M]${ there exists }$\theta _q \in \I$
{ such that } $\beta _{\theta _q } (\Delta _{q+1} ,0)\le \Delta 
_{q+2}$. { Furthermore, let $T_q :[\Delta _q ,\Delta _{q+1} )\to \mR_+ $ be given continuous, separated-from-zero (\ie, $0<T_{\min } =\mathop {\min }\limits_{0\le q\le M} \{\,\mathop {\inf }\limits_{\Delta _q \le s<\Delta _{q+1} } \{\,T_q (s)\,\}\,\}<+\,\infty $), bounded functions. }
\end{assumption}
Assumption \ref{ass2} holds, for instance, if $\beta _{\theta _q } (s,r)\le \beta _i (s,r)$ for all $r\in [0,T_q  (s))$, $s\in [\Delta _q ,\Delta _{q+1} )$, $i\in \I$, which implies that the output trajectories of the system $\theta _q $, starting off in the interval $[\Delta _q ,\Delta _{q+1} )$, converge the fastest to zero as compared to the outputs of any other system of the family (\ref{eq1}). In such  sense, this introduces domains of ``local optimality'' for each system of (\ref{eq1}); therefore, by ensuring the ``right'' switching between the locally optimal systems it is possible to guarantee fast (suboptimal) convergence to zero for the output of (\ref{eq1}). 
	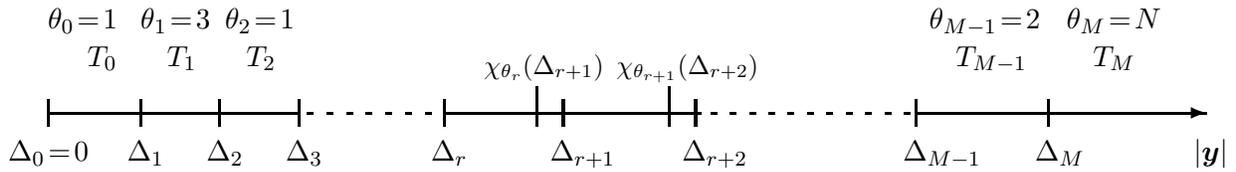
\begin{figure}[h]
\ \\[1cm]
\mbox{}\hspace{-2cm}
  \begin{picture}(0,0)
    \put(75,0){\thicklines\line(1,0){95}}
    \thicklines\dashline[+30]{3}(143,0)(220,0)
    \put(75,-5){\thicklines\line(0,1){10}}
    \put(110,-5){\thicklines\line(0,1){10}}
    \put(140,-5){\thicklines\line(0,1){10}}
    \put(170,-5){\thicklines\line(0,1){10}}
    \put(60,-18){$\Delta_0\!=\!0$}
    \put(105,-18){$\Delta_1$}
    \put(135,-18){$\Delta_2$}
    \put(165,-18){$\Delta_3$}
    \put(75,32){$\theta_0\!=\!1$}
    \put(110,32){$\theta_1\!=\!3$}
    \put(142,32){$\theta_2\!=\!1$}
    \put(90,18){$T_0$}
    \put(120,18){$T_1$}
    \put(150,18){$T_2$}
  \end{picture}
\hspace{5cm}  
  \begin{picture}(0,0)
    \put(75,0){\thicklines\line(1,0){97}}
    \put(75,-5){\thicklines\line(0,1){10}}
    \put(110,-5){\thicklines\line(0,1){15}}
    \put(120,-5){\thicklines\line(0,1){10}}
    \put(160,-5){\thicklines\line(0,1){15}}
    \put(170,-5){\thicklines\line(0,1){10}}
    \put(165,-18){$\Delta_{r+2}$}
    \put(140,15){\small $\chi_{\theta_{r+1}}(\Delta_{r+2})$}
    \put(115,-18){$\Delta_{r+1}$}
    \put(90,15){\small $\chi_{\theta_r}(\Delta_{r+1})$}
    \put(70,-18){$\Delta_r$}
  \end{picture}
\hspace{6cm}  
  \begin{picture}(0,0)
    \put(75,0){\thicklines\vector(1,0){110}}
    \thicklines\dashline[+30]{3}(-10,0)(80,0)
    \put(75,-5){\thicklines\line(0,1){10}}
    \put(125,-5){\thicklines\line(0,1){10}}
    \put(70,-18){$\Delta_{M-1}$}
    \put(120,-18){$\Delta_{M}$}
    \put(80,32){$\theta_{M-1}\!=\!2$}
    \put(132,32){$\theta_{M}\!=\!N$}
    \put(90,18){$T_{M-1}$}
    \put(142,18){$T_M$}
    \put(180,-18){$|\bfy|$}
  \end{picture}
\mbox{}\\[2mm] 
\caption{Illustrative example when Assumption \ref{ass2} holds}
  \label{fig1}
\end{figure}}
      A graphical illustration of Assumption \ref{ass2} is presented in Fig. \ref{fig1}. Assume that the set $I=\{\,1,...,N\,\}$ with $N>1$. Then, a possible partition of ${\mR}_+ $ depending on possible values of $\vert {\rm {\bfy}}\vert$ is shown above. For each interval $[\,\Delta _q ,\Delta _{q+1} \,)$, $0\le q<M$ the number $\theta _q \in \I$ and the function $T_q :[\,\Delta _q ,\Delta _{q+1} \,)\to {\mR}_+ $ are given. According to Fig. \ref{fig1} on the intervals $[\,0,\Delta _1 \,)$ and $[\,\Delta _2 ,\Delta _3 \,)$ the system $\dot {\bfx}={\rm {\bff}}_1 (\,{\rm {\bfx}},{\rm {\bf d}}\,)$ from family (\ref{eq1}) provides ``optimal'' performance in terms of output speed of convergence, on the interval $[\,\Delta _1 ,\Delta _2 \,)$ the system $\dot {\bfx}={\rm {\bff}}_3 (\,{\rm {\bfx}},{\rm {\bf d}}\,)$ possesses the same property and so on. Dwell-time 
functions $T_q :[\,\Delta _q ,\Delta _{q+1} \,)\to {\mR}_+ $, $0\le q<M$ are used to ensure the  existence of delays between switches and right-continuity of the switching signal $i(t)$. In this way each system $\theta _q \in \I$ is active for, at least, $T_q $ units of time.
\begin{remark}
  \label{Remark1} 
Note that one can define functions $T_q :[\Delta _q ,\Delta _{q+1} )\to 
\mR_+ $ as solutions of 
\begin{eqnarray*}
&\beta _{\theta _0 } (\Delta _1 ,T_0 (s))=0.5\beta _{\theta _0 } 
(\Delta _1 ,0),
\quad
s\in [\Delta _0 ,\Delta _1 );
&\\
&  \beta _{\theta _q } (s,T_q (s))=\chi _{\theta _{q-1} } (\Delta _q ) 
&\\
& s\in [\Delta _q ,\Delta _{q+1} ),\ q\in [0,\, M].
&\end{eqnarray*}
Such function $T_q $ estimates the time that is required for the output trajectories of the system $\theta _q $,  to converge to the interval $[\Delta _{q-1} ,\Delta _{q} )$ from the interval $[\Delta _q ,\Delta _{q+1} )$; the first function $T_0$ is thus artificially imposed. One may replace these functions with any other $T_q :[\Delta _q ,\Delta _{q+1} )\to \mR_+ $, $0<q\le M$ admitting the same requirements of continuity, boundedness and that they are separated from zero. For example, a possible choice of $T_q$ for all $0<q\le M$ and for all $s\in [\Delta _q ,\Delta _{q+1} )$ is one which solves:
\[
\beta _{\theta _q } (s,T_q (s))+\gamma (D_q )=\Delta _{q+1} ,
\]\[
\max \{\,0,\Delta _{q+1} -\beta _{\theta _q } (\Delta _q ,0)\,\}<\gamma 
(D_q )<\Delta _{q+1} \,.
\]
In this case,  the functions $T_q $ majorate the amount of time that the output trajectories of system $\theta _q $ remain in the interval $[\Delta _q ,\Delta _{q+1} )$ while under the influence of disturbances with amplitude smaller (with respect to the norm operator $\mathcal S$) than $D_q$. That allows one to exhibit additional switches caused by disturbances smaller in norm than $D_q$. 
\end{remark}

Our control objective is to design a supervisor that generates a piecewise constant switching signal $i:\mR_+ \to \I$ (continuous from the right) such that the closed-loop system is IOS with respect to the output ${{\bfy}}$, the input ${{\bf d}}$ and norm $\S$. Additionally, it is required to use the system $\theta _q $ in the domain  $[\Delta _q ,\Delta _{q+1} )$ for each $0<q\le M$ which is assumed to provide local optimality. We address such problem by designing two types of supervisors: dwell-time and hysteresis based. 

\section{Supervisory output control}\label{sec:main}
\subsection{Dwell-time supervisor}
\label{section4}

Informally speaking, the dwell-time supervisor is defined as follows: consider a set of IOS systems labeled $\theta_q$ with $q \in [0,\ldots,M]$ and a given partition of the non-negative real line as described in Assumption \ref{ass2} and illustrated in Fig. \ref{fig1}. After a given instant $t_j$ at which the output trajectory lays in the interval $[\Delta_j,\Delta_{j+1})$ the next switching is determined as the earliest moment, {\em passed a dwell time}, such that the output in norm belongs to a partition $[\Delta_k,\Delta_{k+1})$. Note that the $k$th and the $j$th intervals are in no particular order with respect to each other and are not necessarily contiguous; {\em e.g.}, the output trajectory $|\bfy(t)|$ may pass from the interval $[\Delta_1,\Delta_{2})$ to the interval $[\Delta_5,\Delta_{6})$ or viceversa by `crossing' intermediate intervals during the dwell time. Once the switching instant is identified, the switching rule (the supervisor) takes the value corresponding to the interval in which the switch occurs; {\em i.e.}, the $k$th mode is selected. To avoid undesirable infinitely fast switching, besides the (variable) dwell-time,  rather subintervals strictly contained in $[\Delta_j, \Delta_{j+1})$ are considered.
 
More precisely, the dwell-time supervisor is defined as follows:
\begin{subequations}\label{eq3}
  \begin{align}
    t_0 &= 0,\ i(t_0 )=r,\ r\in \{0,\ldots, M\},\ 
 | \bfh(\bfx(t_0)) | \in \left[ \Delta _r,\,\Delta _{r+1}\right); \\
\nonumber 
t_k' &:= \mathop {\arg \inf }\limits_{t\ge t_j +T_{i(t_j )} } 
   \left\{\, |\bfh( \bfx(t) ) | \in \left[\Delta _k ,\chi _{\theta _k} (\Delta _{k+1})\,\! \right)\,\right\},\ 
k\in \{0,..,M\,\} \big\backslash \{i(t_j )\} \\
\label{eq3:b}%
t_{j+1} & = \min_{ k\in \{0,..,M\,\} \big\backslash \{i(t_j )\} } \left\{ t_k' \right\};\\
i(t_{j+1})&=k \mbox{ such that } |\bfh(\bfx(t_{j+1}))| \in [\Delta _k ,\chi _{\theta _k } (\Delta _{k+1} ))  \subset \left[\Delta _k ,\Delta _{k+1} \right).
  \end{align}
\end{subequations}
The dwell-time depends on the output trajectories \ie, $T_{i(t_j)} =T_{i(t_j )} (|\bfh(\bfx(t_j ))| )$, $t_j $ where $j=1,2,3,\ldots$ are switching times, $j$ is the index of the last switch and the signal $i(t)$ is piecewise constant for all $t$ such that $ |\bfh(\bfx(t_{j+1}))| \in \N=\bigcup\limits_{q=0}^{M-1} {[\chi _{\theta _q } (\Delta _{q+1} ),\Delta _{q+1} )}$. The definition of the set $\N$ prevents from a fast switching phenomenon (chattering regime) and it plays the role of hysteresis --{\em cf. } \cite{M,HLM,HM}. The second mechanism that prevents the system from chattering is {\em output}\footnote{In this respect we mention \cite{PSM} where dwell time depends on the whole state as opposed to an output of the system.} dependent dwell-time $T_{i(t_j )} $. Since by construction the functions $T_q $, $0\le q\le M$ are bounded and separated from zero the system (\ref{eq2}), with supervisor (\ref{eq3}), has dwell-time $T_q(s)\geq T_{\min}$ and the system undergoes a finite number of switches on any finite-time interval.
\begin{remark}
\label{rmk2}
Opposite to the supervisor considered in the state regulation problem --\cf \cite{E1,ASTPRI}, in the present context, the hysteresis set $\N$ is not sufficient, in
general, to guarantee a finite number of switches. This is because the output regulation set $\N$ may be non compact which stymies the estimation of the system's behavior on the set \eg\,, it is possible that a system's state trajectory (in norm) escapes to infinity  while the output remains in $\N$. In section \ref{sec:hysteresis} a hysteresis supervisor is studied in further detail.
\end{remark}
\begin{remark}
\label{rmk3}
The function $T_0 $ for the supervisor (\ref{eq3}), can be chosen identically equal to zero. In such situation, the system has average dwell-time $T_{\min } $ with 
$N_0 =2$. Indeed, according to (\ref{eq3}), $i(t_j )\ne i(t_{j+1} )$ so the  
system $\theta _0 $ may become active each second time only. Therefore, the 
time interval between two switches is bigger than $T_{\min }$; in this 
case, the  time $T_{\min } $ is calculated for $0<q\le M$. 
\end{remark}
\begin{remark}
  Using Assumption \ref{ass1} one can prove IOS stability for the switched system 
(\ref{eq2}) by the applying dwell-time or average-dwell time approaches; for this, it is required that the dwell-time $T_{\min } $ be sufficiently large --\cf \cite{al:IFAC08-1} for 
details). However, as it was shown in \cite{al:IFAC08-1} such approach is applicable only for 
locally exponentially stable systems (\ref{eq1}). Additionally, taking 
large values for the average dwell time or dwell-time constants implies that all 
functions $\beta _i $ from Assumption \ref{ass1} decrease to zero rather 
slowly (by construction of $T_{\min }$) which is not desirable from 
a practical point of view. Such restriction is opposite to our requirements since it is desirable to apply locally optimal controls in their domains. 
\end{remark}
\begin{theorem} 
\label{Theorem1}
Let Assumptions \ref{ass1} and \ref{ass2} hold. Then, the system (\ref{eq2}) with supervisor (\ref{eq3}), measurable disturbances $\bfd\in M_{\mR^m}$ and initial conditions $t_0 = 0$, $\bfx_0\in\mR^n$,\\
(i) is forward complete and, for all $t\geq t_0$,
  \begin{subequations}
        \label{UGSbound}
  \begin{eqnarray}
    \label{UGSbound:a}
  |\bfy(t,\bfx_0,{\rm {\bf d}})| & \le & \bar\beta(\max\{2\Delta_M, |\bfh(\bfx_0)|\}, 0)  + \bar\gamma(\S[{\rm {\bf d}},0,t])\\ 
    \label{UGSbound:b}
 \bar\beta(s,0) &:=& \sup_{i\in\mathcal I}\beta_i(s,0)\\
    \label{UGSbound:c}
 \bar\gamma(s) & := & \gamma(s) + \beta_{\theta_M}(2\gamma(s),0)\,;
  \end{eqnarray}
  \end{subequations}
(ii) furthermore, if the system dynamics undergoes a finite number of switches over infinite time \ie, if there exists $N<\infty$ such that $\{t_k\}\to t_N < \infty$ (define $t_{N+1} := +\infty$) then, there exists a continuous function $\bar\beta^{\scriptsize \bfd}:\mR_+\times\mR_+\to\mR_+$ such that $\bar\beta^{\scriptsize \bfd}(\cdot,s)$ is strictly increasing for each $s$ and $\bar\beta^{\scriptsize \bfd}(r,\cdot)$ is strictly decreasing to zero for each $r$, such that\footnote{That is, $\bar \beta^{\bfd}$ is ``of class $\mathcal K\mathcal L$'' with the exception that $\bar\beta^{\bfd}(0,\cdot)$ may be different from zero. } 
  \begin{eqnarray}
\label{finalbnd}|\bfy(t,\bfx_0,{\rm {\bf d}})|& \le & \bar\beta^{\scriptsize \bfd}(\vert {{\bfh}}({{\bfx}}_0 )\vert,t)+\gamma (S\,[{\rm {\bf d}},0,\infty])\,;
  \end{eqnarray}
\noindent (iii) finally, if ${\rm {\bf d}}(t)\equiv 0$ (with no assumption on the number of switchings) the output trajectories ${\rm {\bfy}}(t)$ satisfy
\begin{equation}
  \label{d0convergencebound}
|\bfy(t,\bfx_0,{\rm {\bf d}})| \leq \bar\beta^{\bf 0}(|\bfh(\bfx_0)|,t) 
\end{equation}
where $\bar\beta^{\bf 0}\in\cK\cL$. 
\end{theorem}
\begin{proof}[of (i)]
On any time interval $[T_s ,\,T_e )$ with $T_e >T_s \ge 0$ a finite number of switches $N_{\left[ {\,T_s ,T_e } \right)} $ with the upper estimate 
$$N_{\left[ {\,T_s ,T_e } \right)} \le \frac{N_0 +(T_e -T_s)}{T_{\min }} $$ 
may occur. Between switches, the system's dynamics is continuous and is equivalent to that of a system from the family (\ref{eq1}) and which is forward complete for $i\in I$ fixed  arbitrarily. Since the signal $i(t)$ remains constant over  $[T_s ,\,T_e )$ the solutions of system  (\ref{eq2}), (\ref{eq3}) are continuous and are defined at least locally on the same interval. From continuity of solutions of \rref{eq2}, \rref{eq3} and forward completeness of $\dot\bfx = \bff_i(\bfx,\bfd)$ with fixed $i$, it follows that solutions of \rref{eq2}, \rref{eq3} are also defined at $T_e$. The same arguments hold for any interval $[T_s, T_e)$ therefore the switched system \rref{eq2}, \rref{eq3} is forward complete.

We prove next that the bound \rref{UGSbound} holds. From (\ref{eq3}) we have $i(0)=r$ where $r\in [0,\, M]$ and  $\vert {{\bfh}}({\rm  {\bfx}}_0)\vert \in [\Delta _r,\,\Delta _{r+1}) $. According to  Assumption \ref{ass1}, in this case,
\begin{equation}
\label{bndy:t0t1}
  \vert {{\bfy}}(t,{{\bfx}}_0 ,{{\bf d}})\vert \,\,\le \beta 
_{\theta _r } (\vert {{\bfh}}({{\bfx}}_0 )\vert ,t)+\gamma (S\,[{\rm {\bf d}},0,t_1 ]) \quad \forall\, t\in [0,t_1 )
\end{equation}
where $t_1$ is the first switching time and, according with (\ref{eq3}), $t_1 $ is when the output trajectory belongs to the interval $[\Delta _{i(t_1 )} ,\chi _{\theta _{i(t_1 )} } (\Delta _{i(t_1 )+1} ))$ with  $ i(t_1 )\in [0,\, M]$ and $i(t_1 )\ne i(0)$. If $t_1 = +\infty$ the proof ends replacing the bound in \rref{UGSbound} with 
\[
\vert {{\bfy}}(t,{{\bfx}}_0 ,{{\bf d}})\vert \,\,\le \beta 
_{\theta _r } (\vert {{\bfh}}({{\bfx}}_0 )\vert ,0)+\gamma (S\,[{\rm {\bf d}},0,\infty ]) \quad \forall\, t\geq 0\,.
\]
Otherwise, if $t_1<+\infty$, by continuity of solutions, continuity of $\bfh$ and forward completeness the solutions may be continued up to $t_1$ and 
\begin{equation}
 \label{kleenex}
\vert {{\bfy}}(t_1)\vert \,\,\le \chi _{\theta _{i(t_1 )} } (\Delta _{i(t_1 )+1})\,.
\end{equation}
Reconsidering the solutions with new ``initial'' condition $t_1$, $x(t_1)$ we obtain, from the definition of the supervisor \rref{eq3},
\[
\vert {{\bfy}}(t,{{\bfx}}(t_1) ,{{\bf d}})\vert \le  \beta _{\theta _{i(t_1)} } (\vert {{\bfy}}(t_1)\vert ,t-t_1) + \gamma (S\,[{\rm {\bf d}},t_1,t_2 ])\\
\quad
t\in [t_1 ,t_2 )\,.
\]
Using \rref{kleenex} and the definition of $\chi _{\theta _{i(t_1 )}}$ in the expression above, we obtain 
\[
\vert {{\bfy}}(t,{{\bfx}}(t_1) ,{{\bf d}})\vert \le\Delta _{i(t_1 )+1} +\gamma (S\,[{\rm {\bf d}},t_1,t_2 ]),\ t \in [t_1,t_2)\,.
\]
By definition, at time instant $t_2 $ the system's output reaches for the interval $[\Delta _{i(t_2 )} ,\chi _{\theta _{i(t_2 )} } (\Delta _{i(t_2 )+1} ))$ with $ i(t_2 )\in [0,\, M]$ and $i(t_2 )\ne i(t_1 )$ and $  \vert {{\bfy}}(t_2 )\vert \le \chi _{\theta _{i(t_2 )} } (\Delta _{i(t_2 )+1} )$. Following similar arguments as above we see that 
  \begin{align*}
\vert {{\bfy}}(t,\bfx(t_2),\bfd)\vert &\le
\Delta _{i(t_2 )+1} +\gamma (S\,[{\rm {\bf d}},t_2,t_3 ]),\quad t\in [t_2,t_3)
  \end{align*}
Repeating the previous steps we obtain, for arbitrary time instants $t_j $ for any $j>1$ 
  \begin{align}
\label{bndy:tjtj+1}
\vert {{\bfy}}(t,\bfx(t_j),\bfd)\vert &\le
\Delta _{i(t_j)+1} +\gamma (S\,[{\rm {\bf d}},t_j,t_{j+1} ]),\quad t\in [t_j,t_{j+1})\\
\label{bndy:tj}  \vert {{\bfy}}(t_j)\vert & \le  \chi _{\theta _{i(t_j)} } (\Delta _{i(t_j)+1} )\,.
  \end{align}
Now, if at some step $k$ the system's output (in norm) belongs to the last interval $[\Delta _M ,\Delta _{M+1} )$ with $\Delta _{M+1} =+\,\infty $ then, the maximum output amplitude may be estimated by
\begin{eqnarray}
\label{bndy:tM}
  \vert {{\bfy}}(t)\vert & \le & \beta _{\theta _M } (\vert {\rm 
{\bfy}}(t_k )\vert ,0)+\gamma  (S\,[{\rm {\bf d}},t_{k},t_{k+1} ]),
\quad t\in [t_k ,t_{k+1} )\,.
\end{eqnarray}
According with \rref{bndy:tjtj+1} the possible values of ${{\bfy}}(t_k )$ are bounded by\footnote{Where, in view of the definition of the supervisor --\cf \rref{eq3:b} we exclude the value $i(t_k)+1=M+1=\infty$. }
\[
\vert {{\bfy}}(t_k )\vert \,\,\le \Delta _M +\gamma (S\,[{\rm {\bf d}},0,t ]),
\]
Therefore, for all $t\in [t_k ,t_{k+1} )$,
\[
\vert {{\bfy}}(t)\vert \,\,\le \beta _{\theta _M } (\,\Delta _M + \gamma (S\,[{\rm {\bf d}},0,t ]),0\,)+\gamma (S\,[{\rm {\bf d}},0,t ] ).
\]
By induction, the latter holds for any integer $k\geq 1$ hence, for all $t\ge t_1 $. Using the ``triangle'' inequality $\beta_{\theta_M}(a+b,0)\leq \beta_{\theta_M}(2a,0) + \beta_{\theta_M}(2b,0)$ and the bound \rref{bndy:t0t1} we obtain the estimate \rref{UGSbound}.

\noindent {\em Proof of (ii).} We derive next a bound for the output trajectories which leads to \rref{finalbnd}. Following the definition of the supervisor \rref{eq3} we have \rref{bndy:t0t1}. 
According with \rref{eq3}, $t_1$ is the first switching time \ie, when the norm of the output trajectory belongs to the interval $[\Delta _{i(t_1 )} ,\chi _{\theta _{i(t_1 )} } (\Delta _{i(t_1 )+1} ))$ with  $ i(t_1 )\in [0,M]$ and $i(t_1 )\ne i(0)$. If $t_1=+\infty$ the proof ends noting that \rref{bndy:t0t1} implies \rref{finalbnd}.

Otherwise, if $t_1< +\infty$ by continuity of solutions, of $\bfh$, forward completeness, and the definition of the dwell-time function $T_i(t_1)$ --\cf Assumption \ref{ass2}, we have 
\begin{equation}
 \label{bndy:t1}
\vert {{\bfy}}(t_1 )\vert \le \beta _{\theta _r } (\vert {{\bfh}}({{\bfx}}_0 )\vert ,T_{\min } )+\gamma (S\,[{\rm {\bf d}},0,t_1 ])\,.
  \end{equation}

Under similar arguments let $t_2$ be the second switching instant, that is the first time such that $|\bfy(t)|\in[\Delta_{i(t_2 )} ,\chi _{\theta _{i(t_2 )} } (\Delta _{i(t_2 )+1}))$ with $0\le i(t_2 )\le M$ and $i(t_2 )\ne i(t_1 )$. We have 
  \begin{equation}
\label{bndy:t1t2}
\vert {{\bfy}}(t)\vert \le \beta 
_{\theta_{i(t_1)} } (|\bfy(t_1)|,t-t_1)+\gamma (S\,[{\rm {\bf d}},t_1,t_2])\quad \forall\, t\in [t_1,t_2)
  \end{equation}
and, if $t_2 < +\infty$,
\begin{equation}
 \label{bndy:t2}
\vert {{\bfy}}(t_2)\vert \le \beta _{\theta_{i(t_1)}} (|\bfy(t_1)|,T_{\min}) + \gamma (S\,[{\rm {\bf d}},t_1,t_2])\,.
  \end{equation}
Using \rref{bndy:t1t2} in \rref{bndy:t1} we obtain 
\begin{eqnarray}
\label{bnd:t0t2}
  \vert {{\bfy}}(t)\vert &\le& \beta_{\theta_{i(t_1)} } \big(\beta _{\theta _r } (\vert {{\bfh}}({{\bfx}}_0 )\vert ,T_{\min } )+\gamma (S\,[{\rm {\bf d}},0,t_1 ])\,,t-t_1\big)+\gamma (S\,[{\rm {\bf d}},t_1,t_2])\quad \forall\, t\in [t_1,t_2)\\
  \vert {{\bfy}}(t_2)\vert &\le& \beta_{\theta_{i(t_1)} } \big(\beta _{\theta _r } (\vert {{\bfh}}({{\bfx}}_0 )\vert ,T_{\min} )+\gamma (S\,[{\rm {\bf d}},0,t_1 ])\,,T_{\min}\big)+\gamma (S\,[{\rm {\bf d}},t_1,t_2])\,.
\end{eqnarray}
Similar arguments hold for any $k\geq 1$ and all $t\in [t_{k-1}, t_k)$ \ie, 
\begin{eqnarray}
  \label{bndy:t0tk}
\vert {{\bfy}}(t)\vert & \le & \beta_{\theta_{i(t_{k-1})} } \big(\beta_{\theta_{i(t_{k-2})} }\circ\big( \cdots\, \circ\big(\vartheta_0^{\scriptsize\bfd}(|\bfh(x_0)|),T_{\min}\big)+\cdots + \gamma (S\,[{\rm {\bf d}},t_{k-2},t_{k-1}])\,\big), t-t_{k-1}\big)\nonumber \\ && \quad +\gamma (S\,[{\rm {\bf d}},t_{k-1},t_{k}])\\
\vartheta_0^{\scriptsize\bfd}(s) &:= & \beta _{\theta _r } (s,T_{\min})+\gamma (S\,[{\rm {\bf d}},0,t_1 ])\,.
\end{eqnarray}
By assumption there exists $N<\infty$ such that $\{t_k\}\to t_N < \infty$ and $t_{N+1} := +\infty$; from this and \rref{bndy:t0tk} it follows that, for all $t\in[t_k, \infty)$,  
\begin{eqnarray}
  \label{tildebeta}
\vert {{\bfy}}(t)\vert & \le & \widetilde\beta^{\scriptsize \bfd}(\vert {{\bfh}}({{\bfx}}_0 )\vert,t-t_k) 
+\gamma (S\,[{\rm {\bf d}},t_{k},\infty])\ \ \ \mbox{}\\
  \label{bndy:tkinfty}
\tilde\beta^{\scriptsize \bfd}(s,r) & := & \beta_{\theta_{i(t_{k})} } \big(\beta_{\theta_{i(t_{k-1})}}\overbrace{\circ\big( \cdots\, \circ}^{N \mbox{\footnotesize \ times}}\big(\vartheta_0^{\scriptsize\bfd}(s),T_{\min}\big)+\cdots + \gamma (S\,[{\rm {\bf d}},0,\infty])\,\big), r\big)\,.
\end{eqnarray}
The function  $\widetilde\beta^{\scriptsize \bfd}(\cdot,r)$ is strictly increasing for each fixed $r$ and $\bar\beta^{\scriptsize \bfd}(s,\cdot)$ is strictly decreasing to zero for each fixed $s$.

Define $\kappa^{\scriptsize \bfd}(s)$ as the right-hand side of \rref{UGSbound:a} \ie, $\kappa^{\scriptsize \bfd}(s):= \bar\beta(\max\{2\Delta_M, s\}, 0)  + \bar\gamma(\S[{\rm {\bf d}},0,\infty])$. Then, $\bar\beta^{\scriptsize \bfd}$ may be defined as any continuous function satisfying the following: $\bar\beta^{\scriptsize \bfd}(\cdot,r)$ is strictly increasing for each $r\geq0$; $\bar\beta^{\scriptsize \bfd}(s,\cdot)$ is strictly decreasing for each $s\geq 0$;
\begin{eqnarray*}
  \bar\beta^{\scriptsize \bfd}(s,t_N) & = &  \kappa^{\scriptsize \bfd}(s);\\
  \bar\beta^{\scriptsize \bfd}(s,t) & \geq &  \widetilde\beta^{\scriptsize \bfd}(s,t-t_N) \quad \forall\, t\geq t_N;\\
    \bar\beta^{\scriptsize \bfd}(s,t) & :=& \beta^*(s,t) \in\cK\cL \quad \forall \, t\in [0,t_N]
\end{eqnarray*}
such that $ \beta^*(s,t_N)= \kappa^{\scriptsize \bfd}(s) $.

\noindent {\em Proof of (iii).} Finally, consider $\bfd\equiv 0$ and let the initial conditions be $i(0)=r$ with $r\in [0,\, M]$ where $| {{\bfh}}({{\bf x}}(0))| \in [\Delta _r,\, \Delta _{r+1})$. If $ \vert {{\bfh}}({{\bf x}}(0))\vert \in [\chi _{\theta _r } (\Delta _{r+1} ), \Delta _{r+1}) $, the output trajectory $|\bfy(t)|$ may, in general, reach the interval $[\Delta _{r+1} ,\chi _{\theta _{r+1} } (\Delta _{r+2} ))$. On the other hand, if $\vert {{\bfh}}({{\bf x}}(0))\vert \in [\Delta _r ,\chi _{\theta _r } (\Delta _{r+1} ))$ for any $r$ we have, by assumption, $|\bfy(t)|\leq \beta_{\theta_r}(|\bfy(0)|,0) $ which implies that $|\bfy(t)|<\Delta _{r+1}$ \ie, the output trajectory (in norm) cannot reach higher intervals than $[\Delta _{r+1} ,\chi _{\theta _{r+1} } (\Delta _{r+2} ))$ and may only decrease for at least $T_{\theta_r}$ units of time. Therefore, there exists $t'$ such that the system's output trajectories reach the lower next interval $[\Delta _{r-1} ,\chi _{\theta _{r-1} } (\Delta _r ))$. Repeating the reasoning we conclude that $|\bfy(t)|$ may only continue decreasing to zero (right to left on the real line --\cf Figure \ref{fig1}) which implies that the system \rref{eq1} can undergo only a finite number of switches over $[0,\infty)$. The proof ends.  
\end{proof}

Thus system (\ref{eq2}), (\ref{eq3}) has bounded output trajectories for any bounded disturbances and that the system's output converges to zero in the case of that there are no disturbances. Opposite to the case where stability of the switched system follows for large enough values of dwell-time, here dwell-time $T_{\min } >0$ can be chosen arbitrarily small provided that there is a finite number of switches over any finite time interval.

To take into account IOS systems (\cf Definition \ref{Definition1}) one needs additional requirements on its detectability properties. For example, in \cite{JIATEEPRA} it was proved that IOS systems with input-output-to-state stability property (a variant of robust detectability property) is input-to-state stable. In terms of the $\S$ norm previously defined we have the following propositions which follow as corollaries of Theorem \ref{Theorem1}.

\begin{corollary}
\label{Corollary1} {Let all conditions of Theorem \ref{Theorem1} be satisfied and let $T_q$ satisfy}
\begin{equation}
\label{eq5}
\beta _{\theta _0 } (\Delta _1 ,T_0 (s))=0.5\,\Delta _1 ,
\quad
s\in [\,\Delta _0 ,\Delta _1 );
\end{equation}
\begin{equation}\label{seis}
  \beta _{\theta _q } (s,T_q (s))=\chi _{\theta _{q-1} } (\Delta _q ),
\quad
s\in [\,\Delta _q ,\Delta _{q+1} ),\quad 0<q\le M\,.
\end{equation}
{Then, the time} $T_{0.5\,\Delta _1 } $ {that is required for $|\bfy(t)|$ to converging into the set }$\{\vert {\rm {\bfy}}\vert \,\,\le 0.5\,\Delta 
_1\} $ {for the case }${\rm {\bf d}}(t)\equiv 0$ {can be estimated as follows (from (\ref{eq3}) }$i(0)=r)$:
\begin{equation}
\label{eq6}
T_{0.5\,\Delta _1 } \le \sum\limits_{k=0}^r {T_k (\Delta _{k+1} )} .
\end{equation}
{If} $\beta _{\theta _q } (s,T_q (s))+\gamma (D_q )=\chi _{\theta 
_{q-1} } (\Delta _q )$, $0<\gamma (D_q )<\chi _{\theta _{q-1} } 
(\Delta _q )$, $s\in [\,\Delta _q ,\Delta _{q+1} )$ {for} $0<q\le M$ {and }$\beta 
_{\theta _0 } (\Delta _1 ,T_0 (s))+\gamma (D_0 )=0.5\,\Delta _1 
$, $0<\gamma (D_0 )<0.5\,\Delta _1 $, $s\in [\,\Delta _0 ,\Delta _1 
)$, {then estimate (\ref{eq6}) holds under disturbances} ${\rm {\bf d}}\in M_{\mR^m} ${ such that}
    \begin{equation}
      \vert {\rm {\bfy}}(t)\vert \in [\,\Delta _q ,\Delta _{q+1} ),
\quad
q\in [0, M],\ t\in [\,t_1^q ,t_2^q ) \quad \Rightarrow  \quad S\,[\,{\rm {\bf d}},t_1^q ,t_2^q 
]\le D_q\,
    \end{equation}
\end{corollary}
\begin{proof}
It follows using Equations (\ref{eq3}) the definition of $T_q $ and invoking Theorem \ref{Theorem1}.\end{proof}

The relation \rref{seis} defined $T_q(\cdot)$ implicitly; for instance, for $\beta_{\theta_q}:=kse^{-r}$ we obtain $T_q(s):=\ln(ks)-\ln(\chi _{\theta _{q-1} } (\Delta _q ))$ while for $\beta(s,T_q(s))\propto 1/T_q(s)$ we obtain $T_q(s)\propto s$. Estimate (\ref{eq6}) provides upper estimation on finite time of practical stabilization of the system with respect to set where $\vert {\rm {\bf y}}(t)\vert \,\,\le 0.5\,\Delta _1 $.

\begin{corollary}
  \label{Corollary2} {Let all conditions of Theorem \ref{Theorem1} be satisfied and expressions (\ref{eq5}), (\ref{seis}) be valid. Assume that, for all }$q\in (0,\, M]$ {and} $s\in [\,\Delta _q ,\Delta _{q+1} )$, there exists 
$ \,{t}'\in [\,T_q (s),T_q (s)+T_{q-1} (\chi _{\theta _{q-1} 
} (\Delta _q )))$
such that 
$\beta _{\theta _q } (s,t)\ge \beta _{\theta _{q-1} } (\chi _{\theta 
_{q-1} } (\Delta _q ),t-T_q (s))$ {for all} $t\in [\,{t}',T_q 
(s)+T_{q-1} \{\,\chi _{\theta _{q-1} } (\Delta _q )\,\})$.

\noindent{Then,} { for the case }${\rm {\bf d}}(t)\equiv 0$ { the output of the system (\ref{eq2}), (\ref{eq3}) has the shortest time of convergence to zero compared to that of any other system }$\theta _q ${ from the family (\ref{eq1}), with initial output values in }$[\,\Delta _q ,\Delta _{q+1} )$.
\end{corollary}
\begin{proof}
Conditions of the corollary imply that starting with system $q$ it is necessary to switch on to system $q-1$ when the output reaches for $\chi _{\theta _{q-1} } (\Delta _q )$ since system $q-1$ reaches for $\chi _{\theta _{q-2} } (\Delta _{q-1} )$ faster than the $q$th system. 
\end{proof}

Corollary \ref{Corollary2} establishes  conditions under which the output of system (\ref{eq2}), (\ref{eq3}) has shorter time of convergence than any fixed system from the initial family (\ref{eq1}) over its domain of ``optimality'' $[\Delta_q,\Delta_{q+1})$. The overall performance, that is for all $t$ is sub-optimal in the sense that it depends on the way local intervals --\cf Assumption \ref{ass2}-- and functions $T_q $ are chosen; for instance, different stability and convergence properties may be obtained. The purpose of the following example is twofold: firstly, to illustrate the utility of Corollary \ref{Corollary2} by choosing the partitions determined by $\{\Delta_q\}$ and secondly, to provide an insight of the difficulty to obtain ``true'' optimality in the choice of the latter intervals and associated dwell-times. 
  \begin{example}
    Consider a switched system of the form (\ref{eq2}), (\ref{eq3}) with $\I=\{\,0,1\,\}$,  \ie, we have only two dynamic modes, each of which is stable in the sense of Definition \ref{Definition2} with functions $\beta_0$ and $\beta_1$. The goal is to find a set of conditions that lead to an improved rate of convergence --optimality here is understood in that sense. We have $M=2$ and the partition in Assumption \ref{ass2} is ${\mR}_+ =[\,0,\Delta _1 )\cup [\,\Delta _1 ,+\,\infty )$ --\cf Figure \ref{fig1}. The fact leading to optimality (or not) is the appropriate choice of the threshold $\Delta_1$. Since we are looking for the switching rule that minimizes the time of convergence let us introduce the functions $\mathcal T_k :{\mR}_+^2 \to {\mR}_+ $ such, that $\beta _k (s,\mathcal T_k (s,\Delta ))=\Delta $, $s\in {\mR}_+ $, $\Delta \in {\mR}_+ $, $k=0,1$ (it is assumed that $\mathcal T_k (s,\Delta )=0$ for $\beta _k (s,0)\le \Delta)$ \ie, $\mathcal T_k$ is the time that takes for the estimate on the output trajectories of system $k$ to reach a given level $\Delta$. For the level of interest here \ie, $\Delta=\Delta_1$ assume first that $\mathcal T_1 (s,\chi _0 (\Delta _1 ))\le \mathcal T_0 (s,\chi _0 (\Delta _1 ))$ for $s\ge \Delta _1 $; in such case, it is reasonable to ``turn on'' dynamic mode 1 as long as the output trajectory remains within the upper interval $[\,\Delta _1 ,+\,\infty )$ hence, we set $\theta_1=1$ and the system dynamics is defined by $\dot {\bfx}={\rm {\bff}}_1 ({\rm {\bfx}},{\rm {\bf d}})$ on that interval. If on the other hand $\mathcal T_0 (s,0.5\,s)\le \mathcal T_1 (s,0.5\,s)$ for $s<\Delta _1 $ it results reasonable to set $\theta_0=0$ hence, the system's dynamics is given by $\dot {\bfx}={\rm {\bff}}_0 ({\rm {\bfx}},{\rm {\bf d}})$ as long as  $|\bfy(t)|\in[\,0,\Delta _1 )$. After the conditions of Corollary \ref{Corollary2}:
\[
\forall \,\,s\ge \Delta _1 
\quad
\exists \,\,{t}'\ge \mathcal T_1 (s,\chi _0 (\Delta _1 )):
\quad
\beta _1 (s,t)\ge \beta _0 (\chi _0 (\Delta _1 ),t-\mathcal T_1 (s,\chi_0 (\Delta _1 ))\quad
\forall \,\,t\ge {t}'
\]
so the switching condition to change from mode 1: $\dot {\bfx}={\rm {\bf f}}_1 ({\rm {\bfx}},{\rm {\bf d}})$ to mode 0: $\dot {\bfx}={\rm {\bff}}_0 ({\rm {\bfx}},{\rm {\bf d}})$, is that $|\bfy(t)|\in [\,0,\Delta _1 )$.  It is important to remark that these conditions are only sufficient and the constant $\Delta _1 $ is, in general, not the threshold that defines the switching rule which leads to locally optimal performance, in the sense of fastest convergence. However, the fact that there exists a finite $\Delta_1$ satisfying the conditions of Corollary \ref{Corollary2} for this particular situation, means that switched control leads, if not in general to {\em optimal} performance, to an improvement. 

To see farther, let us consider the following optimization problem: for the case of two systems, let $1\gg\varepsilon>0$ be a given tolerance level with respect to which to measure $convergence$ \ie, let us assume that it is required to find the shortest time for the output trajectories to satisfy $|\bfy(t,\bfx_0,0)|\leq \varepsilon$. Then,  
\begin{equation}
\label{deltaep}  \Delta _\varepsilon (s)=\mathop {\arg \min }\limits_{\Delta \le s} 
\{\,T_1 (s,\Delta )+T_0 (\Delta ,\varepsilon )\,\}
\end{equation}
provides the optimal switching threshold we are looking for. If the solution of \rref{deltaep} is constant \ie, if $\Delta_\varepsilon (s)=\widetilde {\Delta }$, then we set $\Delta _1 =\widetilde {\Delta }$; if $\Delta _\varepsilon (s)$ varies with the initial state value then the choice $T_1 (s)=T_1 (s,\Delta _\varepsilon (s))$ ensures that the output trajectory has an optimal convergence to the level $\varepsilon $ for all output trajectories of the switched system (\ref{eq2}), (\ref{eq3}). 
\end{example}
It may be apparent from this discussion that the difficulty of the latter optimization problem increases geometrically with respect to the number of systems, $N$, and partitions, $M$. Yet, satisfactory results may be obtained following sensible considerations on a case-by-case basis as it is further illustrated through particular applications in Section \ref{sec:appl}.

\subsection{Hysteresis supervisor}
\label{sec:hysteresis}
We now trade the dwell-time condition by a hysteresis assumption. Opposite to the previous case in which switches may occur to {\em any} mode provided that a minimal time passes, we now assume that switches occur as soon as the output value (in norm) leaves a determined interval modulo a hysteresis zone to prevent infinite switches over finite intervals. In particular, switching may occur from mode $q$ to modes $q-1$ or $q+1$ only. The hysteresis supervisor is defined as follows:
\begin{subequations}\label{eq7}
  \begin{align}
    t_0 &= 0,\ i(t_0 )=r,\ r\in \{0,\ldots, M\},\ 
 | \bfh(\bfx(t_0)) | \in \left[ \Delta _r,\,\Delta _{r+1}\right); \\
\nonumber t_k' &:= \mathop {\arg \inf }\limits_{t\ge t_j } 
   \left\{\, |\bfh( \bfx(t) ) | \in \left[\Delta _k ,\chi _{\theta _k} (\Delta _{k+1})\,\! \right)\,\right\},\ 
k\in \{\,i(t_j )-1,i(t_j )+1\,\} \\
t_{j+1} & = \min_{ k\in \{\,i(t_j )-1,i(t_j )+1\,\} } \left\{ t_k' \right\};\\
i(t_{j+1})&=k \mbox{ such that } |\bfh(\bfx(t_{j+1}))| \in [\Delta _k ,\chi _{\theta _k } (\Delta _{k+1} ))  \subset \left[\Delta _k ,\Delta _{k+1} \right).
  \end{align}
\end{subequations}
where $t_j $, with $j=1,2,3,...$, are switching times; $j$ is the number of the last switch and the signal $i(t)$ has constant value in the so-called ``hysteresis'' set $\N=\bigcup\limits_{q=0}^{M-1} {[\,\chi _{\theta _q } (\Delta _{q+1} ),\Delta _{q+1} )} $.
  \begin{theorem}
    \label{Theorem2} Let Assumptions \ref{ass1} and \ref{ass2} hold. Then, for the  system (\ref{eq2}) with supervisor (\ref{eq7}), measurable disturbances $\bfd\in M_{\mR^m}$ and initial conditions $t_0 = 0$, $\bfx_0\in\mR^n$, items (i)--(iii) of Theorem \ref{Theorem1} hold{ with 
    \begin{equation}\label{thm2:UGSbound:a}
        |\bfy(t,\bfx_0,{\rm {\bf d}})|  \le  \bar\beta(\max\{\Delta_M, |\bfh(\bfx_0)|\}, 0)  + \gamma(\S[{\rm {\bf d}},0,t])\quad \forall \,t\geq t_0,
    \end{equation}
    instead of \rref{UGSbound:a}. }
  \end{theorem}
\begin{remark}\label{rmk50}
It is important to stress that in the previous statement Assumption \ref{ass1} needs $not$ to hold $globally$ \ie, each of the systems in \rref{eq2} is not required to be SIIOS for all initial states and all measurable disturbances. It is sufficient that each system $\theta_q$ is SIIOS for initial states in the set corresponding to output values where the system is to be active. To illustrate this, consider Figure \ref{fig1} and system $\theta_1$; for the hysteresis supervisor it is enough that the system $\theta_1$ be SIIOS for all $\bfx_0$ such that $|\bfh(\bfx_0)|\in [\Delta_0, \Delta_1] \cup [\Delta_2, \Delta_3]$. This is of obvious interest if system $\theta_1$ corresponds to a plant in closed loop with a controller which guarantees stability in the large (\ie, in a ``large'' specified domain of attraction, subset of $\mR^n$) or even locally. In Section \ref{sec:lorenz} we present a case-study presenting such characteristic. 
  \end{remark}
  \begin{proof}[of Theorem \ref{Theorem2}]
From (\ref{eq7}) $i(t_0 )=r$, $t_0 =0$ where $\vert {\rm {\bfh}}({\rm {\bfx}}_0)\vert \in [\Delta _r, \,\Delta _{r+1}) $, $ r\in [0,\, M]$. According to Assumption \ref{ass1} system $\theta _{r} $ from family (\ref{eq1}) is forward complete so estimate
\[
\vert {\rm {\bfy}}(t,{\rm {\bfx}}_0 ,{\rm {\bf d}})\vert \,\,\le \beta 
_{\theta _{i(t_0 )} } (\vert {\rm {\bfh}}({\rm {\bfx}}_0 )\vert 
,t)+\gamma (S\,[\,{\rm {\bf d}},0,t_1 ])
\]
holds also for the trajectories of system \rref{eq2}, \rref{eq7} over $t\in [\,t_0 ,t_1 )$. From (\ref{eq7}) the time instant $t_1 $ is the time instant when the output trajectory (in norm) enters in the interval $[\,\Delta _{i(t_1 )} ,\chi _{\theta _{i(t_1 )} } (\Delta _{i(t_1 )+1} ))$ with $i(t_1 )\in \{\,\theta _{i(t_0 )} -1,\theta _{i(t_0)} +1\,\}$. Assume that $t_1 <+\,\infty $ (if $t_1 $ is infinite, then the system is clearly forward complete and moreover the bound \rref{finalbnd} holds) then from forward completeness of $\dot \bfx = \bff_{\theta_r}(\bfx,\bfd)$ and considering ${\rm {\bf d}}\in M_{\mR^m}$ we obtain that there exist $X_0 \in {\mR}_+ $ and $D_0 \in {\mR}_+ $ such that the solutions of \rref{eq2}, \rref{eq7} are defined over $[t_0,t_1)$ and, moreover, $\vert \vert {\rm {\bfx}}\vert \vert _{[\,t_0 ,t_1 )} \,\,\le X_0 $, $\vert \vert {\rm {\bf d}}\vert \vert _{[\,t_0 ,t_1 )} \,\,\le D_0 $. Hence, 
\[
F_0 =\mathop {\sup }\limits_{{\rm {\bfx}}\in N,\,\vert {\rm {\bfx}}\vert 
\,\,\le X_0 ,\,\vert {\rm {\bf d}}\vert \,\,\le D_0 } \vert {\rm {\bf 
f}}_{\theta _{i(t_0 )} } ({\rm {\bfx}},{\rm {\bf d}})\vert \ < \ \infty
\]
which implies that for $|\bfy(t)|$ generated by system $\theta_{i(t_0 )}$ to reach adjoining intervals $\{\,[\,\Delta _{\theta _{i(t_0)} -1},\,\Delta _{\theta _{i(t_0 )} } )$ and $[\,\Delta _{\theta _{i(t_0 )} +1} ,\Delta _{\theta _{i(t_0 )} +2} )\,\}$ it is necessary a time proportional to the maximum ``speed'' $F_0$ and 
$t_1 - t_0$ units of time. Moreover, since $F_0<\infty$ necessarily $t_1>t_0$ therefore, the solutions of system (\ref{eq2}), (\ref{eq7}) are defined and are continuous over $[t_0,t_1]$. Reconsidering the initial time to be $t_1$ and \rref{eq2}, \rref{eq7} we obtain that the solutions of the latter are defined for all $t\in [\,t_1 ,t_2)$ and satisfy
\[
\vert {\rm {\bfy}}(t,\bfx(t_1),t_1)\vert \,\,\le \beta _{\theta _{i(t_1 )} } 
(\vert {\rm {\bfy}}(t_1 )\vert ,t-t_1)+\gamma (S\,[\,{\rm {\bf d}},t_1 ,t_2 ])
\]
and there exists finite numbers $X_1 \in {\mR}_+ $, $D_1 \in {\mR}_+ $ and $F_1 \in {\mR}_+ $ such that
\[
\vert \vert {\rm {\bfx}}\vert \vert _{[\,t_1 ,t_2 )} \,\,\le X_1 ,
\quad
\vert \vert {\rm {\bf d}}\vert \vert _{[\,t_1 ,t_2 )} \,\,\le D_1 ,
\quad
F_1 =\mathop {\sup }\limits_{{\rm {\bfx}}\in \mathcal N,\,\vert {\rm {\bfx}}\vert 
\,\,\le X_1 ,\,\vert {\rm {\bf d}}\vert \,\,\le D_1 } \vert {\rm {\bf 
f}}_{\theta _{i(t_1 )} } ({\rm {\bfx}},{\rm {\bf d}})\vert 
\]
hence $t_2 >t_1 $. Repeating these arguments for arbitrary $j>0$ it is 
possible to prove the existence of $X_j \in {\mR}_+ $, $D_j \in {\mR}_+ $ and $F_j 
\in {\mR}_+ $ with properties
\[
\vert \vert {\rm {\bfx}}\vert \vert _{[\,t_j ,t_{j+1} )} \,\,\le X_j ,
\quad
\vert \vert {\rm {\bf d}}\vert \vert _{[\,t_j ,t_{j+1} )} \,\,\le D_j ,
\quad
F_j =\mathop {\sup }\limits_{{\rm {\bfx}}\in \mathcal N,\,\vert {\rm {\bfx}}\vert 
\,\,\le X_j ,\,\vert {\rm {\bf d}}\vert \,\,\le D_j } \vert {\rm {\bf 
f}}_{\theta _{i(t_j )} } ({\rm {\bfx}},{\rm {\bf d}})\vert ,
\]
so $t_{j+1} >t_j $. This implies right continuity of the switching signal $i(t)$ and forward completeness of the system (\ref{eq2}), (\ref{eq7}).

The rest of the proof follows along similar lines as the proof of Theorem \ref{Theorem1}{: the estimates \rref{bndy:tjtj+1} and \rref{bndy:tj} continue to hold {\em verbatim} and if at $t_k$ we have $|\bfy(t_k)|\in [\Delta_M,\Delta_{M+1})$ then \rref{bndy:tM} holds true. In view of \rref{eq7} we have either $|\bfy(t_k)| = |\bfy(t_0)|$ and $k=0$ or $|\bfy(t_k)|=\Delta_M$ (this is because according to the definition of the hysteresis supervisor, there is no dwell time); the bound \rref{thm2:UGSbound:a} follows. The proof of   statement (ii) follows as in Theorem \ref{Theorem1} by replacing $T_{\min}$ with $0$; finally, the proof of (iii) follows the same arguments as Theorem \ref{Theorem1}.}
  \end{proof}
  \begin{corollary}
    \label{Corollary3} {Let all conditions of Theorem \ref{Theorem2} hold and let $T_q $, for $0\le q\le M$, be defined by (\ref{eq5}), (\ref{seis})}. {Then, the time} $T_{0.5\,\Delta _1 } $ {that is required for $|\bfy(t)|$ to converge to the set }$\{\vert {\rm {\bfy}}\vert \,\,\le 0.5\,\Delta _1\} $ {for the case }${\rm {\bf d}}(t)\equiv 0$, {satisfies}
\[
T_{0.5\,\Delta _1 } \le T_r (\Delta _{r+1} )+T_{r+1} (\Delta _{r+1} 
)+\sum\limits_{k=0}^r {T_k (\Delta _{k+1} )} \,.
\]
  \end{corollary}

  \begin{proof}
If ${\rm {\bfh}}(\vert {\rm {\bfx}}(0)\vert )\in [\,\chi_{\theta _r } (\Delta _{r+1} ),\Delta _{r+1} )$ then $|\bfy(t)|$ may, in general, reach the interval $[\,\Delta _{r+1} ,\Delta _{r+2} )$ even in the case ${\rm {\bf d}}(t)\equiv 0$. From (\ref{eq7}) and Assumptions \ref{ass1}, \ref{ass2}  it follows that $|\bfy(t)|$ may not increase and the trajectories of the system (\ref{eq2}), (\ref{eq7}) behave similarly as the trajectories of the system (\ref{eq2}), (\ref{eq3}), (\ref{eq5}), (\ref{seis}). This confirms the corollary. 
  \end{proof}

The advantage of the hysteresis supervisor (\ref{eq7}) over the dwell-time supervisor (\ref{eq3}) lays in the fact that for the latter, during dwell-time, the system's output can in general reach any interval $[\,\chi _{\theta _q } (\Delta _{q+1} ),\Delta _{q+1} 
)$, $q\in[0,M] $ before a switch occurs. In contrast to this, in the case of the hysteresis supervisor (\ref{eq7}) only adjoint intervals may be taken into account which results in a simpler analysis of the system's behavior. Additionally, the supervisor (\ref{eq7}) ensures that only the system $\theta_q $ may become active for output values in the interval $[\,\Delta _q ,\Delta _{q+1})$, while for the dwell-time supervisor (\ref{eq3}) it is possible that the interval is reached during the dwell-time period by any (sub)system from the family (\ref{eq1}). The latter allows to relax Assumption \ref{ass1} for the hysteresis supervisor, as is explained in Remark \ref{rmk50}.

On the other hand, as it is established in Corollaries \ref{Corollary1}, \ref{Corollary2} and discussed below the latter, a proper choice of output depending dwell-time functions $T_q(\cdot)$ can ensure that the system posses additional (desired) stability properties; for instance, the upper estimate (\ref{eq6}) given in Corollary \ref{Corollary1} for supervisor (\ref{eq3}) is better than the estimate provided by the conditions of Corollary \ref{Corollary3} for supervisor (\ref{eq7})). Additionally for appropriate (large) values of dwell-time functions $T_q$ the switched system may admit an SIIOS like estimate  --\cf\cite{al:IFAC08-1}. Another shortage of supervisor (\ref{eq7}) with respect to \rref{eq3} is that the hysteresis-based supervisor allows, in general, for regimes with ``high'' rate of switches.

Finally, it is worth pointing out that an important difference between the hysteresis supervisor proposed above and others published in the literature --\cf \cite{PSM}, is that the hysteresis and partition properties are used without restrictions on dwell time values.

\section{Applications}\label{sec:appl}
\subsection{Hybrid observer design for Lur'e systems}
\label{sec:lure}

The problem of observer design for nonlinear dynamical systems has been one of the centers of  attention of the control community during decades and still has not been completely solved for nonlinear systems --\cf\, \cite{KKK,MARTOM}. Here, we address the observer design problem for Lur'e-type systems via the supervisory control approach proposed above. Consider the system
\begin{equation}
\label{eq8}
{{\dot {\bfx}}}={{\bf A}}({{\bfy}})\,{{\bfx}}+\phi 
({{\bfy}})+{{\bf B}}\,{{\bf d}},
\quad
{{\bfy}}={{\bf C}}\,{{\bfx}}
\end{equation}
where ${{\bfx}}\in R^n$, ${{\bfy}}\in R^p$, ${{\bf d}}\in R^m$ are state, measured output and disturbance vectors correspondingly. This problem is well investigated \eg,  it can be solved via linear systems observers design theory. Indeed, after \cite{F1}, assume that there exists a continuous matrix function ${{\bf K}}:R^p\to R^{n\times p}$ such that for some positive definite matrix ${{\bf P}}$ with dimension $n\times n$ for any ${\rm {\bfy}}\in R^p$ we have 
\begin{equation}
\label{eq9}
{{\bf G}}({{\bfy}})^T{{\bf P}}+{{\bf P}}{{\bf 
G}}({{\bfy}})\le -\alpha {{\bf P}},
\quad
{{\bf G}}({{\bfy}})={{\bf A}}({{\bfy}})-{{\bf 
K}}({{\bfy}}){{\bf C}}
\end{equation}
for some $\alpha >0$. A globally exponentially convergent observer for system (\ref{eq8}) takes form
\begin{equation}
\label{eq10}
{{\dot {\bfz}}}={{\bf A}}({{\bfy}})\,{{\bfz}}+\phi 
({{\bfy}})+{{\bf K}}({{\bfy}})\,[{{\bfy}}-{\rm 
{\bf C}}\,{{\bfz}}],
\end{equation}
where ${{\bfz}}\in R^n$ is a  vector of estimates of ${{\bfx}}$. The observation error ${{\bfe}}={{\bfx}}-{{\bfz}}$ possesses the following dynamics:
\[
{{\dot {\bfe}}}={{\bf G}}({{\bfy}})\,{{\bfe}}+{\rm 
{\bf B}}\,{{\bf d}},
\]
whose origin is exponentially stable for $\bf d\equiv 0$, due to (\ref{eq9}). In the terminology of the paper, the system (\ref{eq8}), (\ref{eq10}) is SIIOS with respect to output ${{\bfe}}$ and input ${{\bf d}}\not\equiv 0$. 

An unclear yet natural question is that of assigning the gain function ${\rm
{\bf K}}$ to ensure the best (trading off, speed and peaking) convergence of
error ${{\bfe}}$ to zero. Assume that the matrix inequality (\ref{eq9}) may be
solved for any $0<\alpha _{\min } \le \alpha \le \alpha _{\max } $ then it is
possible to design a gain function ${{\bf K}}$ to provide ``slow'' ($\alpha
_{\min } )$ or ``fast'' ($\alpha _{\max } )$ speed of error convergence to
zero. Choosing ${{\bf K}}$ for the ``fastest'' solution may, in general, result
in large overshoots (peaking). Such negative behavior is significant for
instance in the context of observer-based designed synchronization --\cf\,
\cite{NIJMAR} or when the state variables ${{\bfz}}$ in (\ref{eq10}), as well as
inputs ${{\bf K}}({{\bfy}})\,[{{\bf y}}-{{\bf C}}\,{{\bfz}}]$, are physically
meaningful.  On the other hand, for small deviations \ie\,, when the error
${{\bfe}}$ stays under some admissible bounds, the application of ``fastest''
gain ${{\bf K}}$ is also undesirable. Thus, for small and large amplitudes of
observation error ${{\bfe}}$ it is desirable to apply a gain function ${{\bf
K}}$ close to $\alpha _{\min } $, while on average amplitudes of ${{\bfe}}$ the
``fastest'' gain ${{\bf K}}$ may provide better performance. This leads to a
natural switching design between ``locally optimal'' controls.

To deal with this problem the proposed approach can be applied. If the system 
(\ref{eq8}) is forward complete for any ${{\bf d}}\in M_{R^m} $, then for any 
triplet $({{\bf K}}_i ({{\bfy}}),{{\bf P}}_i ,\alpha _i 
)$, $1\le i\le N$ such that (\ref{eq9}) is satisfied the system (\ref{eq8}), (\ref{eq10}) is 
forward complete and SIIOS with respect to the output ${{\bfe}}$ and the  input 
${{\bf d}}$ and  
\[
\beta _i (s,r)=\sqrt {\,2\,\frac{\lambda _{\max } ({{\bf P}}_i 
)}{\lambda _{\min } ({{\bf P}}_i )}} \,s\,e^{-0.25\,\alpha _i 
\,r},
\]\[
\gamma (s)=2\,\rho \,\vartheta _{\max } ({{\bf B}})\,s,
\quad
\rho =\mathop {\max }\limits_{1\le i\le N} \left\{\,\frac{\vartheta _{\max } 
({{\bf P}}_i )}{\sqrt {\,\alpha _i } \,\lambda _{\min } ({{\bf 
P}}_i )}\,\right\},
\]
where $\lambda _{\min } ({{\bf P}})$, $\lambda _{\max } ({{\bf P}})$ are minimum
and maximum eigen-values of matrix ${{\bf P}}$ and $\vartheta _{\max } ({{\bf
P}})$ is the maximum singular value of ${{\bf P}}$. Therefore, Assumption
\ref{ass1} holds. Assumption \ref{ass2} is also satisfied for a suitable
partition of the interval $[0,+\,\infty )$: in this case the supervisor
(\ref{eq3}) provides a desired switching between observers with different gains,
which allows to avoid the peaking phenomenon while decreasing ``control energy''
spent on observation (respectively in synchronization). Next, we consider a concrete brief example; for simplicity, we
set $T_i(s)=T_{\min}>0$ that is, we use a constant dwell-time.

\begin{example}
Let the system (\ref{eq8}) be a lossless pendulum \ie,
\[
\begin{array}{l}
 \dot {x}_1 =x_2 ;\,\,\,y=x_1 \,; \\ 
 \dot {x}_2 =-\omega ^2\sin (x_1 )+d\,, \\ 
 \end{array}
\]
where $\omega \in R$ is a known parameter. 

\vbox to 1.8in {\ }
\begin{figure}[h]
  \begin{picture}(0,0)\footnotesize
    \put(0,0){\includegraphics[height=1.8in]{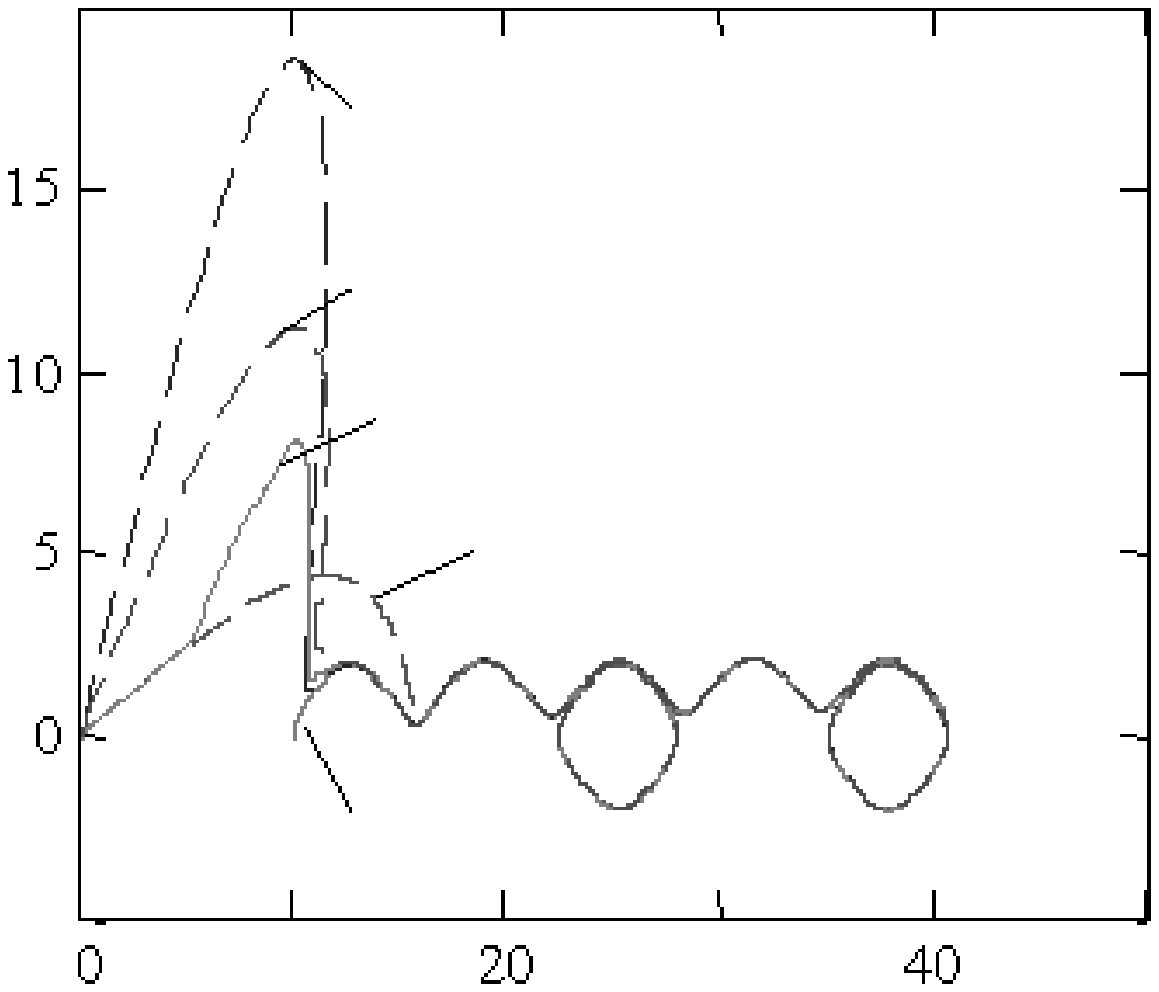}}
    \put(47,90){median} \put(50,72){hybrid} \put(64,55){slow} \put(47,107){fast} 
    \put(43,15){pendulum phase curve} \put(0,105){\begin{rotate}{90}$x_2,\, z_2$\end{rotate}}
    \put(130,-5){$x_1,\ z_1$}
    \put(75,-15){(a)}
    \put(165,0){\includegraphics[height=1.85in]{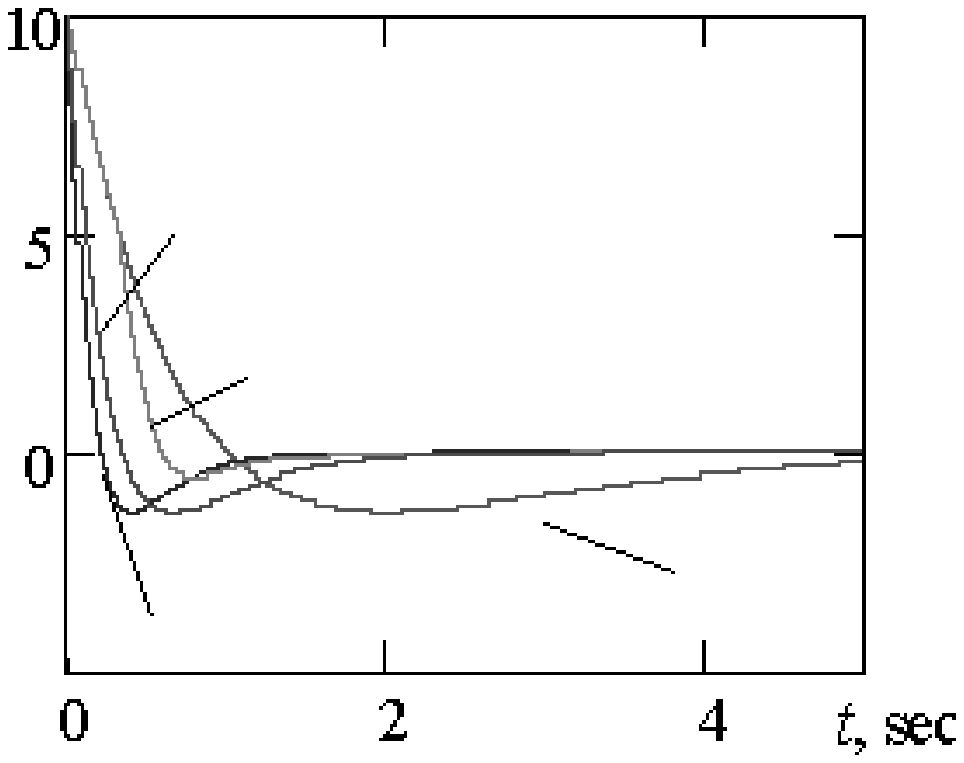}}
    \put(245,-15){(b)}
    \put(196,87){median} \put(208,62){hybrid} \put(277,22){slow}  \put(192,17){fast}
    \put(295,0){\colorbox{white}{$t$ [sec]\ \ \ }}
    \put(165,105){\begin{rotate}{90}$e_1$\end{rotate}}
    \put(330,0){\includegraphics[height=1.82in]{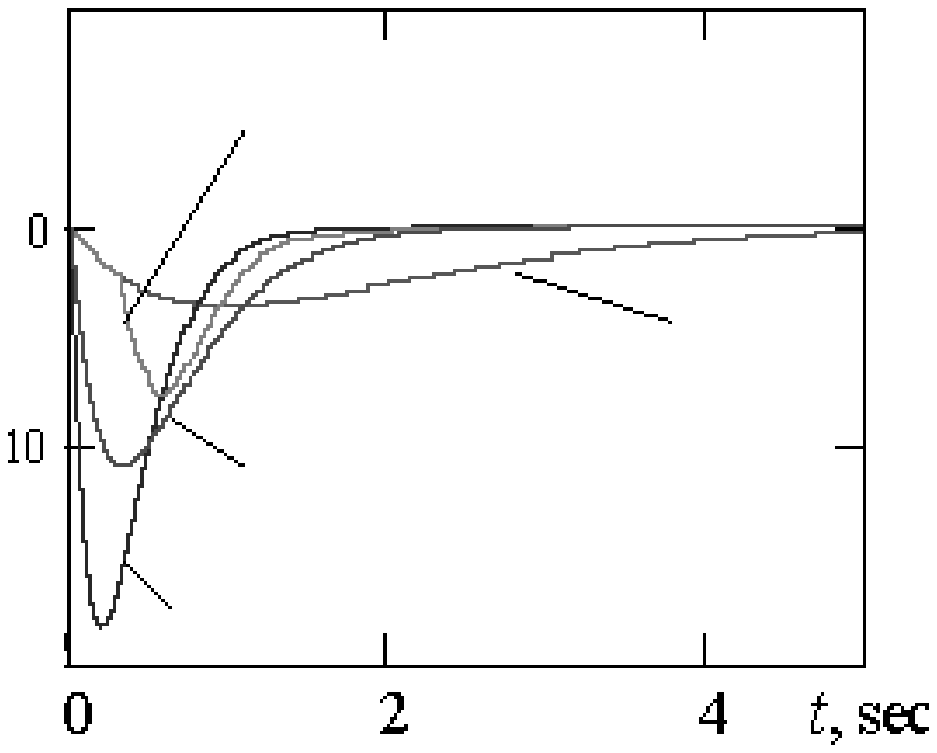}}
    \put(412,-15){(c)}
    \put(372,46){median} \put(377,107){hybrid} \put(441,63){slow} \put(360,17){fast} 
    \put(465,2){\colorbox{white}{$t$ [sec]\ \ \ }}
    \put(330,105){\begin{rotate}{90}$e_2$\end{rotate}}
  \end{picture}
\ \\
\caption{Lur'e system: (a) pendulum and observer trajectories; (b) position estimation errors; (c) velocity estimation errors }
\label{fig2}
\end{figure}

The observer (\ref{eq10}) takes the form
\[
\begin{array}{l}
 \dot {z}_1 =z_2 +k_1 (x_1 -z_1 ); \\ 
 \dot {z}_2 =-\omega ^2\sin (x_1 )+k_2 (x_1 -z_1 )\,, \\ 
 \end{array}
\]
where ${{\bf K}}=[k_1 \,\,\,k_2 ]^T$, $k_1 >0$, $k_2 >0$. Choosing $k_1 =\lambda
_1 +\lambda _2 $ and $k_2 =\lambda _1 \,\lambda _2 $ it is possible to assign
any real negative eigen-values $-\lambda _1 $, $-\lambda _2 $ to the poles of the closed-loop
system. We compute three choices of observer gains for the system $(N=3)$ as
${{\bf K}}_{slow} =[2\,\,\,1]^T$ ($\lambda _1 =\lambda _2 =1)$; ${\rm 
{\bf K}}_{median} =[6\,\,\,9]^T$ ($\lambda _1 =\lambda _2 =3)$; ${\rm 
{\bf K}}_{fast} =[10\,\,\,25]^T$ ($\lambda _1 =\lambda _2 =5)$,
\noindent which provides different speeds of observation error convergence (for $\omega =1$)  --\cf Figs. \ref{fig2}b and \ref{fig2}c. The hybrid observer is constructed as follows: let the following partition of $\mR_+$ be given ($M=3)$:
\[
\Delta _0 =0,
\quad
\Delta _1 =0.1,
\quad
\Delta _2 =2,
\quad
\Delta _3 =5,
\quad
\Delta _4 =+\,\infty ;
\]
with the corresponding operation modes $\theta _0 =median$, $\theta _1 =fast$, $\theta _2 =median$, $\theta _3 =slow$ and constant dwell-time $T_i(s)\equiv T_{\min}:=0.01$sec. In view of the responses depicted in Figs. \ref{fig2}b and \ref{fig2}c we identify slow, median and fast modes. In Fig. \ref{fig2}a one can see the phase curves of the pendulum (\ie, in the space $x_1$--$x_2$) and of the observer (\ie, in the space $z_1$--$z_2$). As one appreciates, setting high observer gains for increased speed entails large overshoots and vice-versa. Based on such observations, to avoid peaking while  conserving a fast convergence rate, we decide to apply the slow observer for large observation errors and to use the fast observer only for sufficiently small values of observation errors. Besides, in the simulation scenario we have added the external disturbance $d(t)=0.05\,\sin (0.3\,t)$; the system's response to the perturbation is appreciated in the loops (closed curves) that appear in Fig \ref{fig2}a \ie, for ``large'' values of the disturbance, the pendulum undergoes full revolutions while for small values of the disturbance, only small oscillations are observed. The performance improvement via the hybrid observer is clear from Fig. \ref{fig2}.
\end{example}


\subsection{Hybrid synchronization of Lorenz oscillators}  
\label{sec:lorenz}

Consider the problem of master-slave synchronization --\cf \cite{PECCAR} of two Lorenz systems:
\begin{equation}
  \label{lo:eq1}
\mbox{master: }\left\{\begin{array}{lcl}
\dot {x}_1 &=&\sigma \,(\,y_1 -x_1 \,)+d_1 (t)\\
\dot {y}_1 &=&x_1 \,(\,\rho -z_1 \,)-y_1 +d_2 (t);\\
\dot {z}_1 &=&x_1 \,y_1 -\beta \,z_1 +d_3 (t)
\end{array}\right.
\end{equation}
\begin{equation}
  \label{lo:eq2}
\mbox{slave: }\left\{\begin{array}{lcl}
\dot {x}_2 &=&\sigma \,(\,y_2 -x_2 \,)\\
\dot {y}_2 &=&x_2 \,(\,\rho -z_2 \,)-y_2 +u_1 \\
\dot {z}_2 &=&x_2 \,y_2 -\beta \,z_2 +u_2 ,
\end{array}\right.
\end{equation}
where $(\,x_i ,y_i ,z_i \,)\in R^3$, $i=1,2$ are the state variables of systems
(\ref{lo:eq1}), (\ref{lo:eq2}); $u_i \in R$, $i=1,2$ are control inputs to
system (\ref{lo:eq2}), ${\rm {\bf u}}=(\,u_1 \,\,u_2 \,)$; $d_i \in R$,
$i=1,2,3$ are disturbances on the system (\ref{lo:eq1}) hence ${\rm {\bf
d}}=(\,d_1 \,\,d_2 \,\,d_3 \,)$ and ${\rm {\bf d}}\in R_{R^3} $; the parameters
$\sigma >0$, $\rho >0$ and $\beta >0$ are identical for both systems. System
(\ref{lo:eq1}) plays a role of ``master'' and system (\ref{lo:eq2}) is the
``slave'', it is assumed that with given disturbing inputs all solutions of
system (\ref{lo:eq1}) are bounded for all $t\ge 0$. We also assume that state
variables $(\,x_i ,y_i ,z_i \,)$, $i=1,2$ of the master, (\ref{lo:eq1}), and the slave, 
(\ref{lo:eq2}), are available from measurements. It is a well known fact that for
$\sigma =10$, $\beta =8/3$ and $\rho =28$ with $d_i =0$, $i=1,2,3$ the system
(\ref{lo:eq1}) is chaotic --\cf\cite{Lo}. We suppose that the disturbances ${\rm {\bf
d}}$ are small enough and they do not destroy the natural strange attractor of the system for the case without the disturbances. The problem of
synchronization of Lorenz systems has seen a lot of attention during the last
decades --\cf \cite{BTWW,Ma,FNA}, partly because its applicability in encoding and secure telecommunication --\cf \cite{KUOOPPSTR}. Briefly, in such context the goal is to generate a chaotic carrier signal by the transmitter (the master system) that may be decoded provided that the receiver (the slave system) synchronizes its dynamics (at least with respect to an output) with the master.

While the synchronization problem may be solved under different control theory viewpoints as observer design --\cf \cite{NIJMAR} or tracking control --\cf \cite{al:LORZAV} we wish to illustrate here how hybrid control may be used to achieve {\em controlled} synchronization while minimizing the input energy. To that end, we define the synchronization errors as $e_1=x_1 -x_2 $, $e_2 =y_1 -y_2 $, $e_3 =z_1 -z_2 $ which obey the dynamics:
\begin{equation}
\label{lo:eq3}
\left\{
\begin{array}{lcl}
  \dot {e}_1 =\sigma (e_2 -e_1 )+d_1 (t)\\
\dot {e}_2 =e_1 \rho -x_1 z_1 +x_2 z_2 -e_2 +d_2 (t)-u_1 \\
\dot {e}_3 =x_1 y_1 -x_2 y_2 -\beta e_3 +d_3 (t)-u_2 .
\end{array}\right.
\end{equation}
We design control inputs that ensure input-to-state stability for system (\ref{lo:eq3}); in this case, systems (\ref{lo:eq1}), (\ref{lo:eq2}) are SIIOS with respect to the synchronization error ${\rm {\bf e}}=(e_1,\, e_2,\, e_3 )^\top$ as output and the input ${\rm {\bf d}}$. 
In particular, we use the hysteresis supervisor \rref{eq7}.

To comply with Assumption \ref{ass1} we first design two controllers that achieve the synchronization goal: the first is a cancellation control law and the second is based on the linearization around the origin. The cancellation law is given by
\begin{equation}
\label{lo:eq4}
u_1 =\left( {\,\rho +\sigma -2\,\sqrt {\,(\,1-\lambda \,)\,\sigma } -z_2 \,} 
\right)\,e_1 ,
\quad
u_2 =y_2 \,e_1 ,
\quad
0<\lambda <1,
\end{equation}
for which system (\ref{lo:eq3}) has Lyapunov function $V(\,{\rm {\bf e}}\,)=0.5\,{\rm {\bf e}}^T{\rm {\bf e}}$ and
\[
\dot {V}\le -\lambda \,\sigma \,e_1^2 -\left( {\,\sqrt {\,(1-\lambda 
)\,\sigma } \,e_1 -e_2 \,} \right)^2-\beta _3 \,e_3^2 +{\rm {\bf e}}^T{\rm 
{\bf d}}.
\]
Clearly the system with such Lyapunov function globally satisfies the Assumption 1.

The second control is local in the sense that it may be applied only if $\vert 
{\rm {\bf e}}\vert \,\,\le \varepsilon $, where $\varepsilon >0$ is given:
\begin{equation}
\label{lo:eq5}
u_1 =\alpha \,e_1 ,
\quad
u_2 =0,
\quad
\alpha >0.
\end{equation}
With control (\ref{lo:eq5}) the system (\ref{lo:eq3}) can be rewritten as follows
\[
{\rm {\bf \dot {e}}}={\rm {\bf A}}\,{\rm {\bf e}}+{\rm {\bf D}}(t),
\quad
{\rm {\bf A}}=\left[ {\,{\begin{array}{*{20}c}
 {-\sigma } \hfill & \sigma \hfill & 0 \hfill \\
 {\rho -\alpha } \hfill & {-1} \hfill & 0 \hfill \\
 0 \hfill & 0 \hfill & {-\beta } \hfill \\
\end{array} }\,} \right],
\quad
{\rm {\bf D}}(t)=\left[ {\,{\begin{array}{*{20}c}
 {d_1 (t)} \hfill \\
 {d_2 (t)-x_1 (t)\,z_1 (t)+x_2 (t)\,z_2 (t)} \hfill \\
 {d_3 (t)+x_1 (t)\,y_1 (t)-x_2 (t)\,y_2 (t)} \hfill \\
\end{array} }\,} \right],
\]
where $\vert \vert {\rm {\bf D}}\vert \vert \,\,\le \,\,\vert \vert {\rm 
{\bf d}}\vert \vert +4\,\varepsilon ^2$ and ${\rm {\bf D}}\in \mR_{\mR^3} $ is a
``new'' bounded disturbance. The system is linear time invariant and the matrix ${\rm {\bf A}}$ is Hurwitz for $\alpha >27$. The system with control (\ref{lo:eq5}) satisfies Assumption \ref{ass1} locally: on $\{\vert {\rm {\bf e}}\vert \,\,\le 
\varepsilon \}$. If $\varepsilon <1$ the system is input-to-state stable with input $ {\rm {\bf D}}$.

A third control input is considered which consists in applying\ldots\, $no$ input \ie,
\begin{equation}
\label{lo:eq6}
u_1 =0,
\quad
u_2 =0.
\end{equation}
The motivation for this control law is that both, master and slave, systems have bounded solutions. For the former \ie,  (\ref{lo:eq1}) it holds by assumption while for the slave system (\ref{lo:eq2}) we know that all trajectories approach a strange attractor which is strictly contained in a compact. 

We can now rewrite the system (\ref{lo:eq3}) as
\[
{\rm {\bf \dot {e}}}={\rm {\bf {A}'}}\,{\rm {\bf e}}+{\rm {\bf 
{D}'}}(t),
\quad
{\rm {\bf {A}'}}=\left[ {\,{\begin{array}{*{20}c}
 {-\sigma } \hfill & \sigma \hfill & 0 \hfill \\
 0 \hfill & {-1} \hfill & 0 \hfill \\
 0 \hfill & 0 \hfill & {-\beta } \hfill \\
\end{array} }\,} \right],
\quad
{\rm {\bf {D}'}}(t)=\left[ {\,{\begin{array}{*{20}c}
 {d_1 (t)} \hfill \\
 d_2 (t)-x_1 (t)\,z_1 (t)+x_2 (t)\,z_2 (t)\hfill\\
\ \ \ \ +\rho \,[\,x_1 (t)-x_2 (t)\,] \hfill \\
 {d_3 (t)+x_1 (t)\,y_1 (t)-x_2 (t)\,y_2 (t)} \hfill \\
\end{array} }\,} \right]
\]
with disturbance ${\rm {\bf {D}'}}\in \mR_{\mR^3} $. Matrix ${\rm {\bf {A}'}}$ 
is Hurwitz and the system satisfies Assumption \ref{ass1} globally.

Thus, we have three systems in \rref{eq2} which are composed of systems
(\ref{lo:eq1}), (\ref{lo:eq2}) with controls (\ref{lo:eq4})--(\ref{lo:eq6}). Let
$I=\{\,1,2,3\,\}$ and consider the partition of $\mR_+$ defined in Assumption \ref{ass2} with $M=3$ and $\theta _0=3$, $\theta _1 =2$, $\theta _2 =1$ and $\theta _3 =3$. The {\em practical} motivation for such a partition is the following: we wish to diminish the amplitude of control energy amplitude therefore, permanent application of the cancellation controller (\ref{lo:eq4}) is not desirable; instead, we would like to switch to the linear local control (\ref{lo:eq5}) for relatively median values of the output errors (hence $\Delta _2 =\varepsilon)$. When synchronization errors are considerably small we may afford to switch off the control action \ie,  ``control'' \rref{lo:eq6} is active for output errors smaller than $\Delta _1 \ll\varepsilon \le 1$. On the other hand, in the case when differences in the state trajectories are ``large'' the use of the  cancellation control law (\ref{lo:eq4}) may lead to large overshoots in control effort which is obviously undesirable since, in particular, may cause actuator saturation. Hence, in this situation we also use control (\ref{lo:eq6}) and let the state trajectories of the forced and unforced Lorenz models (\ref{lo:eq1}) and (\ref{lo:eq2}) converge to their common strange attractor without control effort. In view of the latter, $\Delta _3 $ is chosen proportional to the diameter of a sphere strictly containing the strange attractor which may be computed numerically\footnote{We are not aware of any work computing analytically the ``size'' of the Lorenz attractor.}.

\begin{example}
Let us consider a numerical example. Let $\sigma =10$, $\beta =8/3$ and
$\rho =28$; $\lambda =0.1$ and $\alpha =28$; $\Delta _1 =0.1$, $\Delta _2 =1$
and $\Delta _3 =5$. For illustration, we have performed simulations applying the hysteresis supervisor and the controls (\ref{lo:eq4})--(\ref{lo:eq6}) individually. For the sake of comparison we use the following performance functionals:
\[
J_e =T^{-1}\int\limits_0^T {\vert {\rm {\bf e}}(t)\vert ^2dt} ,
\quad
J_a =10\,T^{-1}\int\limits_{.9\,T}^T {\vert {\rm {\bf e}}(t)\vert ^2dt} 
,
\quad
J_u =T^{-1}\int\limits_0^T {\vert {\rm {\bf u}}(t)\vert ^2dt} ,
\]
where $T>0$ defines the length of simulations' windows. The functional $J_e $ describes the overall quality of synchronization in terms of the integral square error (ISE), the functional $J_a $ gives the ISE for the last tenth part of the simulation window, the functional $J_u $ estimates the input control energy. We use two sets of initial conditions (everywhere initial time is zero):
\begin{equation}
\label{lo:eq7}
x_1 (\,0\,)=0.1,
\quad
y_1 (\,0\,)=z_1 (\,0\,)=0;
\quad
y_2 (\,0\,)=1,
\quad
x_2 (\,0\,)=z_2 (\,0\,)=-1;
\end{equation}
\begin{equation}
\label{lo:eq8}
x_1 (\,0\,)=0.1,
\quad
y_1 (\,0\,)=z_1 (\,0\,)=0;
\quad
y_2 (\,0\,)=10,
\quad
x_2 (\,0\,)=z_2 (\,0\,)=-10,
\end{equation}
which correspond to ``small'' and ``large'' initial deviations of the 
systems (\ref{lo:eq1}) and (\ref{lo:eq2}), and disturbances
\begin{equation}
\label{lo:eq9}
d_1 (t)=5\,\sin (\,0.5\,t\,),
\quad
d_2 (t)=-5\,\cos (\,0.1\,t\,),
\quad
d_3 (t)=2.5\,\sin (\,t\,).
\end{equation}
The simulation results are presented in Table \ref{tab1} for the
case without disturbances and in Table \ref{tab2} for the disturbances
(\ref{lo:eq9}) with $T=30$ sec. As it is possible to conclude from these tables the
supervisory control ensures the best asymptotic performance minimizing the
control energy. While the overall transient is worse for large initial conditions (due to the application of control (\ref{lo:eq6}) which by itself does not solve the problem of synchronization) the ``asymptotic'' quality of synchronization expressed by the index $J_a$ is comparable with those provided by controls (\ref{lo:eq4}), (\ref{lo:eq5}) but with much smaller control effort, again, due to the use of controller (\ref{lo:eq6}).

\begin{table}[htbp]
\begin{center}
\begin{tabular}{|p{77pt}|p{40pt}|p{40pt}|p{40pt}|p{40pt}|p{40pt}|p{40pt}|}
\hline
\raisebox{-1.50ex}[0cm][0cm]{}& 
\multicolumn{2}{|p{80pt}|}{\hspace{8mm}value of $J_e $} & 
\multicolumn{2}{|p{80pt}|}{\hspace{8mm}value of $J_a $} & 
\multicolumn{2}{|p{80pt}|}{\hspace{8mm}value of $J_u $}  \\
\cline{2-7} 
 & 
i.c. (\ref{lo:eq7})& 
i.c. (\ref{lo:eq8})& 
i.c. (\ref{lo:eq7})& 
i.c. (\ref{lo:eq8})& 
i.c. (\ref{lo:eq7})& 
i.c. (\ref{lo:eq8}) \\
\hline
Control (\ref{lo:eq4})& 
0.019& 
1.800& 
0& 
0& 
5.009& 
649.191 \\
\hline
Control (\ref{lo:eq5})& 
0.034& 
2.578& 
0& 
0& 
5.684& 
346.089 \\
\hline
Control (\ref{lo:eq6})& 
541.324& 
365.279& 
518.694& 
350.101& 
0& 
0 \\
\hline
Supervisory control& 
0.034& 
141.329& 
0.009& 
0.009& 
5.436& 
3.548 \\
\hline
\end{tabular}
\end{center}
\caption{Values of performance functionals for the case without disturbances}
\label{tab1}
\end{table}

\begin{table}[htbp]
\begin{center}
\begin{tabular}{|p{77pt}|p{40pt}|p{40pt}|p{40pt}|p{40pt}|p{40pt}|p{40pt}|}
\hline
\raisebox{-1.50ex}[0cm][0cm]{}& 
\multicolumn{2}{|p{80pt}|}{\hspace{8mm}value of $J_e $} & 
\multicolumn{2}{|p{80pt}|}{\hspace{8mm}value of $J_a $} & 
\multicolumn{2}{|p{80pt}|}{\hspace{8mm}value of $J_u $}  \\
\cline{2-7} 
 & 
i.c. (\ref{lo:eq7})& 
i.c. (\ref{lo:eq8})& 
i.c. (\ref{lo:eq7})& 
i.c. (\ref{lo:eq8})& 
i.c. (\ref{lo:eq7})& 
i.c. (\ref{lo:eq8}) \\
\hline
Control (\ref{lo:eq4})& 
0.570& 
2.526& 
0.864& 
0.864& 
71.065& 
675.778 \\
\hline
Control (\ref{lo:eq5})& 
0.597& 
5.181& 
0.335& 
0.335& 
50.030& 
602.633 \\
\hline
Control (\ref{lo:eq6})& 
256.362& 
428.067& 
474.852& 
375.061& 
0& 
0 \\
\hline
Supervisory control& 
0.485& 
34.445& 
0.317& 
0.317& 
45.355& 
39.365 \\
\hline
\end{tabular}
\end{center}
\caption{Values of performance functionals for the case with disturbances (\ref{lo:eq9})}
\label{tab2}
\end{table}

For completeness, some plots are shown in Figures 3 and 4. The former presents the simulation results for the case with disturbances given by (\ref{lo:eq9}) and initial conditions given by (\ref{lo:eq8}) under control (\ref{lo:eq4}). Figure 4 depicts simulation results corresponding to supervisory control under the same conditions. More particularly, we show in Fig. 3a and 4a the phase portraits of the master and slave systems under control (\ref{lo:eq4}) and hysteresis-based supervisory control respectively. In Figs. 3b and 3c we show the control efforts on different scales under control law (\ref{lo:eq4}); these may be compared with the supervisory control input depicted in Fig 4b. As it is appreciated from the plots, the application of $no$ control \ie, (\ref{lo:eq6}) for the large initial errors results in a serious decreasing in the applied control energy $J_u$ and maximal amplitude of $|u(t)|$. The latter has the cost of ``high'' peaks appreciated in the steady-state errors --\cf Fig. 4b; yet the overall steady-state behavior is comparable to that under controller (\ref{lo:eq4}) --\cf Fig. 3c.
\end{example}

\begin{figure}[h]
\vbox to 6.5in{ }
\begin{picture}(0,0)
\put(0,230){%
\begin{minipage}{3in}
{\centerline{\includegraphics[width=2.5in]{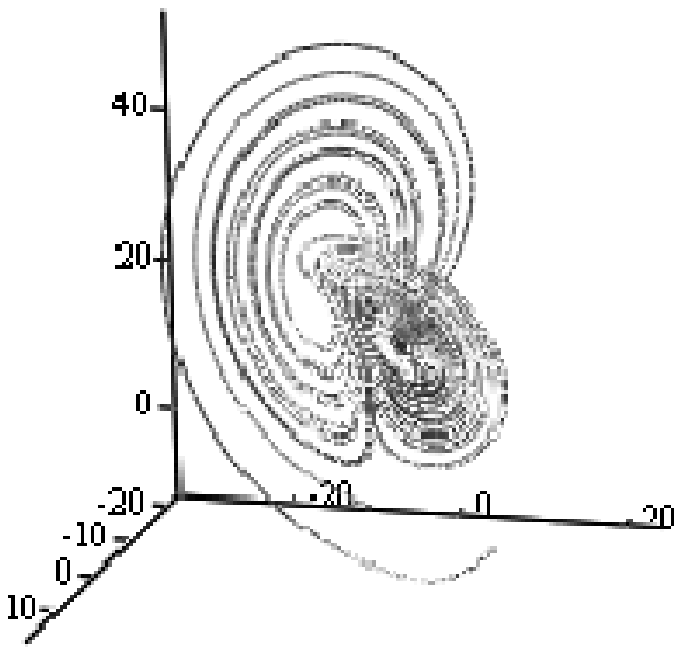}\\[-7mm]}}
{\centerline{(a) Phase portraits}\mbox{}\\[1mm]}
{\centerline{\includegraphics[width=3in]{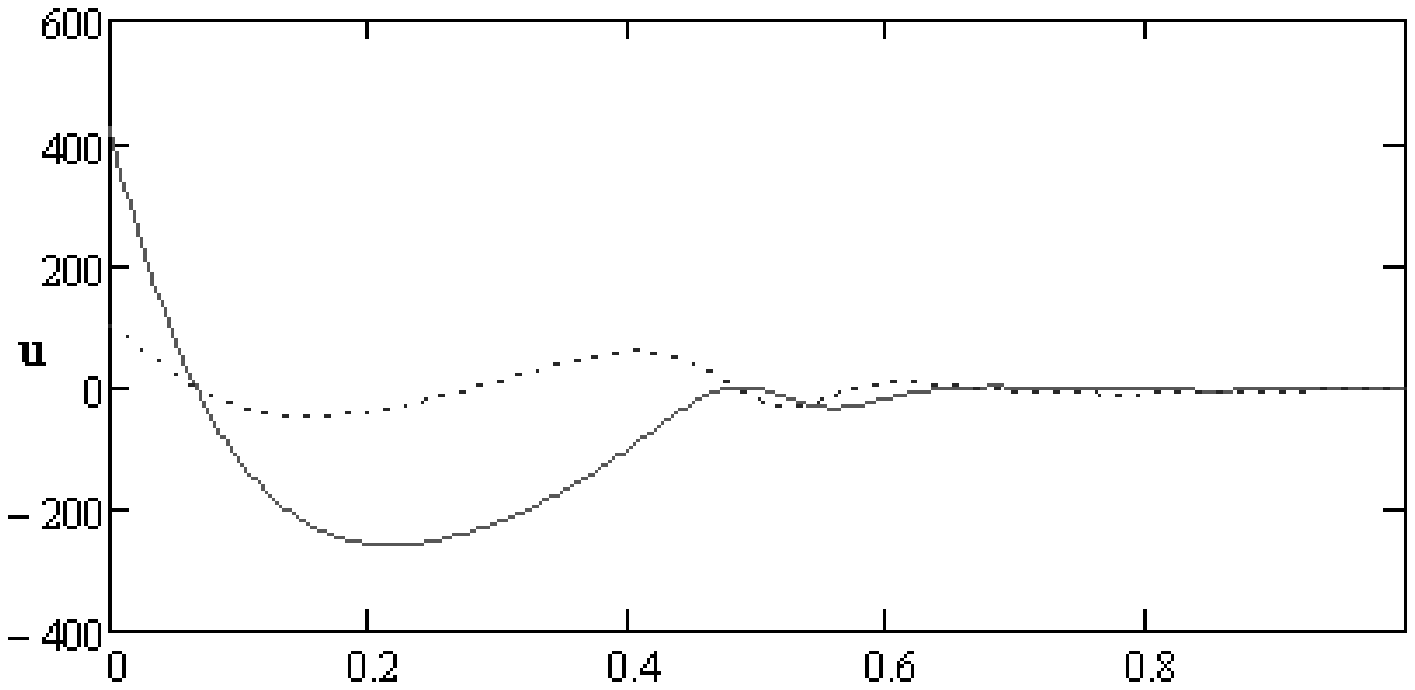}}}
\centerline{(b) Control inputs (zoom of first sec)}\mbox{}\\[3mm]
\centerline{\includegraphics[width=3in]{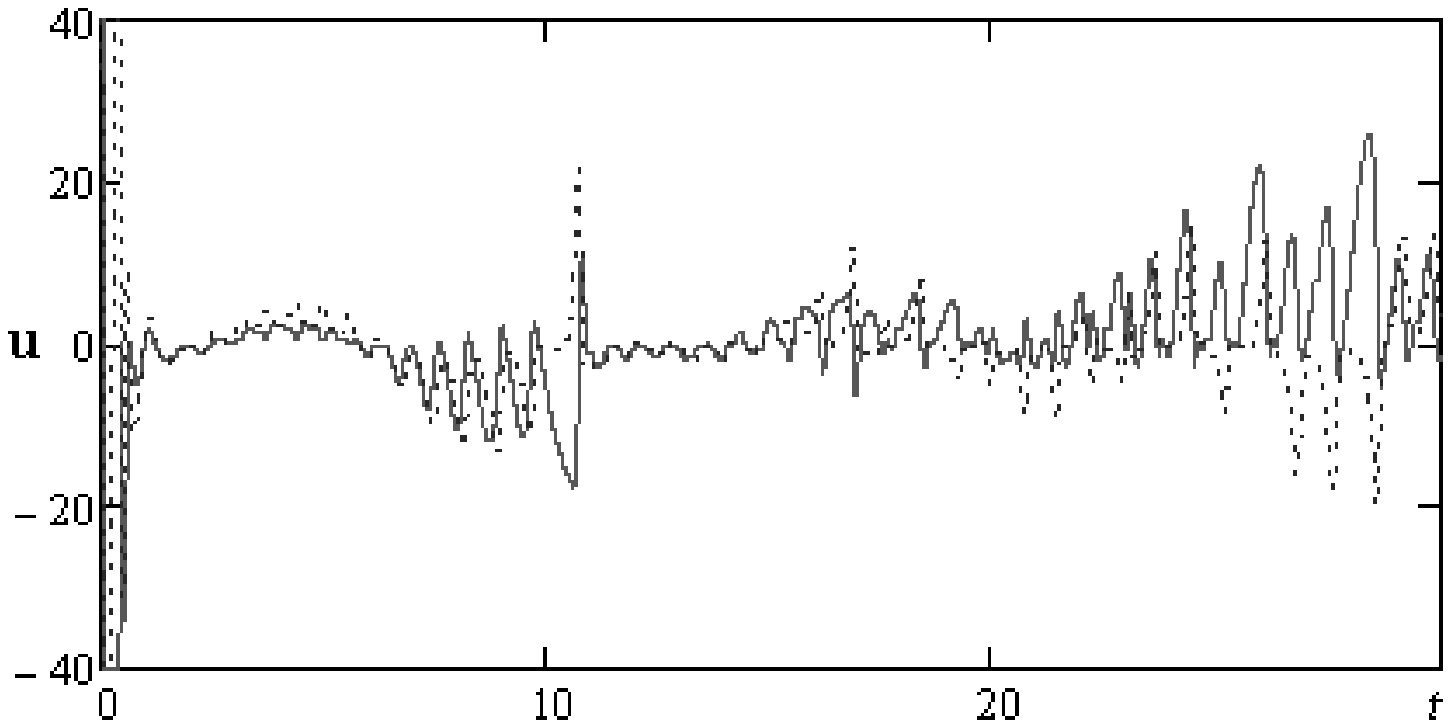}}
\centerline{(c) Control inputs (zoom)}\mbox{}\\[3mm]
\stepcounter{figure}
{Figure \thefigure: Lorenz systems under cancellation control \rref{lo:eq4}}
\hspace{5mm}\begin{picture}(0,0)
  \put(20,350){\dashline[+45]{3}(-10,-10)(15,15)}
  \put(0,330){\small master}
  \put(8,357){\dashline[+45]{3}(80,-20)(65,-15)}
  \put(80,327){\small slave}
\end{picture}
\end{minipage}}
\put(240,263){%
\begin{minipage}{3in}
\centerline{\includegraphics[width=2.5in]{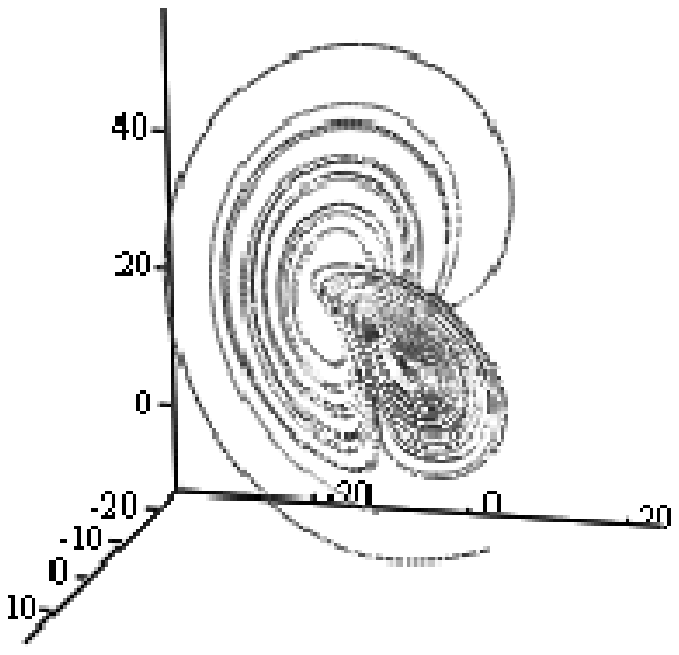}\\[-7mm]}
\centerline{(a) Phase portraits}\mbox{}\\[1mm]
\centerline{\includegraphics[width=3in]{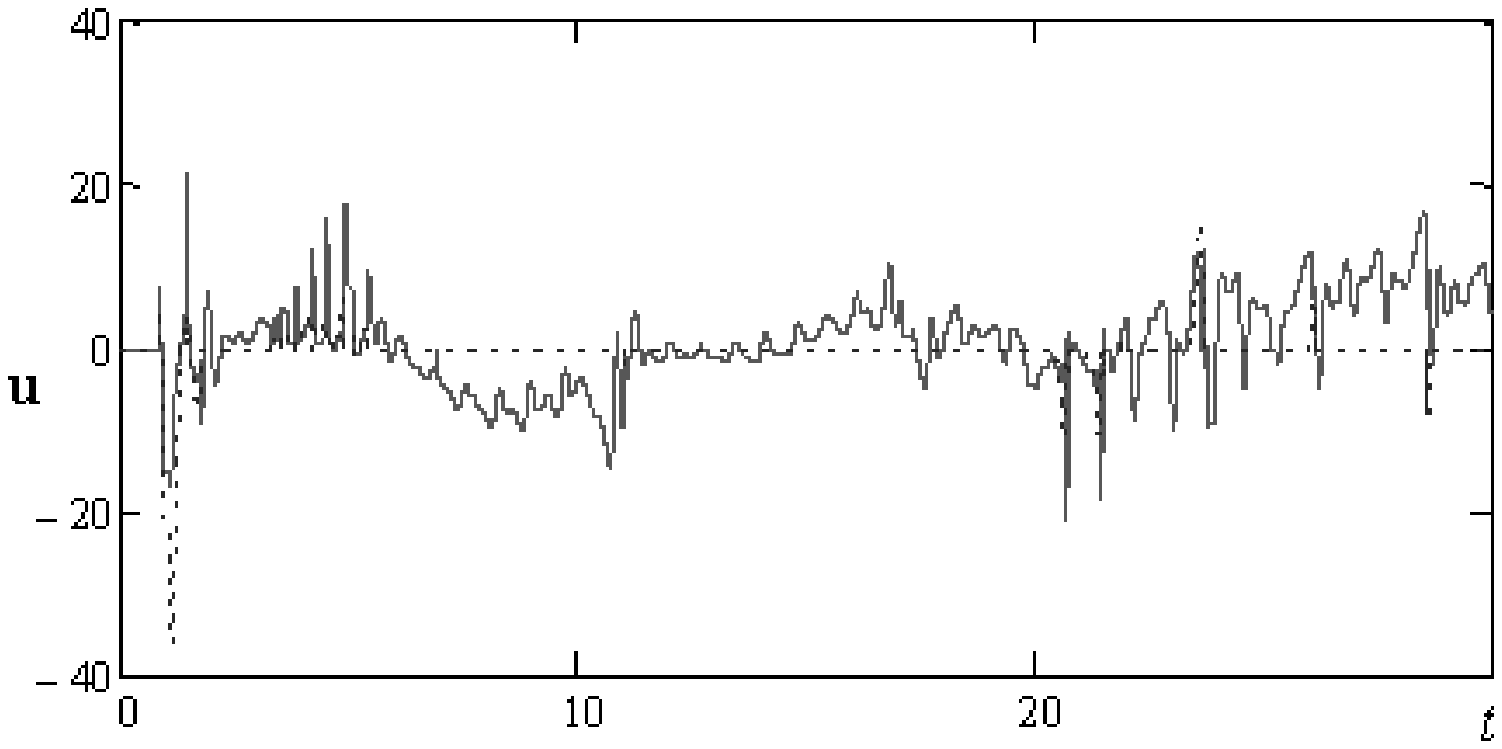}}
\centerline{(b) control inputs}\mbox{}\\[3mm]
\stepcounter{figure}
{Figure \thefigure: Lorenz systems under supervisory control}
\hspace{7mm}\begin{picture}(0,0)
  \put(37,220){\dashline[+45]{3}(-10,-10)(15,15)}
  \put(17,200){\small master}
  \put(25,227){\dashline[+45]{3}(80,-20)(65,-15)}
  \put(97,197){\small slave}
\end{picture}
\end{minipage}
}
\end{picture}
\end{figure}

\section{Conclusion}
\label{sec:conlusion}

We have addressed the problem of establishing input to output stabilization for a family of (switching) nonlinear systems. Two approaches are proposed and analyzed based on dwell-time and hysteresis supervisors. The utility of our approach is illustrated by designing a switching observer for Lur'e-type systems and the performance improvement is stressed for a particular example of a pendulum. A second illustrative application concerns the master-slave synchronization problem of two Lorenz systems.

\small\parskip=1pt
\bibliographystyle{ieeetr}
\bibliography{refs}


\end{document}